\documentclass{amsart}

\usepackage{amssymb}

\usepackage{tikz}

\newtheorem{theorem}{Theorem}[section]
\newtheorem{lemma}[theorem]{Lemma}
\newtheorem{corollary}[theorem]{Corollary}
\newtheorem{proposition}[theorem]{Proposition}

\theoremstyle{definition}
\newtheorem{definition}[theorem]{Definition}

\theoremstyle{remark}
\newtheorem{remark}[theorem]{Remark}

\numberwithin{equation}{section}

\DeclareMathOperator{\Stab}{Stab}
\DeclareMathOperator{\ch}{char}
\DeclareMathOperator{\Gal}{Gal}
\DeclareMathOperator{\Frob}{Frob}
\DeclareMathOperator{\res}{Res}
\DeclareMathOperator{\FPP}{FPP}
\DeclareMathOperator{\Norm}{N}
\DeclareMathOperator{\Aut}{Aut}
\DeclareMathOperator{\Rat}{ {Rat} }
\newcommand{\fp}{ {\mathfrak p} }
\newcommand{\fq}{ {\mathfrak q} }
\newcommand{\fG}{ {\mathfrak G} }
\newcommand{\fP}{ {\mathfrak P} }
\newcommand{\fo}{ {\mathfrak o} }
\newcommand{\cP}{ {\mathcal P} }
\newcommand{\cV}{ {\mathcal V} }
\newcommand{\cL}{ {\mathcal L} }
\newcommand{\cW}{ {\mathcal W} }
\newcommand{\cS}{ {\mathcal S} }
\newcommand{\cH}{ {\mathcal H} }
\newcommand{\lra} {\longrightarrow}
\newcommand{\bA} { {\mathbb A}} 
\newcommand{\F} { {\mathbb F}} 
\newcommand{\bP} { {\mathbb P}} 
\newcommand{\bN} { {\mathbb N}} 
\newcommand{\bQ} { {\mathbb Q}} 
\newcommand{\bC} { {\mathbb C}}
\newcommand{\bR} { {\mathbb R}}

\begin{document}

\title[Iterates of generic polynomials and rational functions]{Iterates of generic polynomials and generic rational functions}


\author{J.~Juul}
\address{
	Department of Mathematics and Statistics\\
	Amherst College\\
	Amherst, MA 01002\\
	USA
}
\email{jamie.l.rahr@gmail.com}
\thanks{NSF grant DMS-1200749}


\date{}

\dedicatory{Dedicated to R.W.K. Odoni}

\begin{abstract}
	In 1985, Odoni showed that in characteristic $0$ the Galois group of the $n$-th iterate of the generic polynomial with degree $d$ is as large as possible. That is, he showed that this Galois group is the $n$-th wreath power of the symmetric group $S_d$. We generalize this result to positive characteristic, as well as to the generic rational function. These results can be applied to prove certain density results in number theory, two of which are presented here. This work was partially completed by the late R.W.K. Odoni in an unpublished paper.
\end{abstract}

\maketitle

Several of the results proven in this paper were stated and proven by R.W.K. Odoni in an unpublished preprint, including the polynomial versions of the Galois theoretic results and the application presented in Section \ref{application}. Although the results and arguments given by Odoni in that manuscript are presented somewhat differently here, his work on this project was invaluable in the completion of this paper.

\section{Introduction}

Given a field $K$ and a rational function $\varphi\in K(x)$, we can form a sequence of fields by adjoining the roots of successive iterates of $\varphi$. The Galois groups of these field extensions have been studied since the 1980's beginning with the work of R.W.K. Odoni \cite{odoni}. The area has seen a recent surge of interest due to its many applications to density questions in number theory \cite{Stoll}\cite{Jones2}\cite{HJM}\cite{JM}\cite{JKMT}. 

In his original paper \cite{odoni}, Odoni showed that if $K$ is a number field, or more generally  a Hilbertian field with characteristic $0$, then for any $n$, most polynomials will have the property that the Galois group of the field extension formed by adjoining the roots of the $n$-th iterate is as large as possible. This follows directly from his result that the generic polynomial over any field of characteristic $0$ has this property.  We generalize this result to rational functions defined over Hilbertian fields with arbitrary characteristic and polynomial functions defined over Hilbertian fields when the degree of the polynomial and the characteristic of the field are not both 2.  A \textit{Hilbertian field} is a field $K$ in which for any irreducible polynomial $f(t_1,\ldots,t_r,x)\in K[t_1,\ldots,t_r,x]$ there exists $a_1,\ldots, a_r\in K$ such that $f(a_1, \ldots, a_r, x)$ is irreducible in $K[x]$. Examples of Hilbertian fields include $\bQ$, number fields, and finite extensions of $k(t)$ for any field $k$ \cite{FJ}.

If $\varphi(x)$ is any rational function in $K(x)$, where $K$ is a field, $\Gal(\varphi(x)/K)$ will denote the Galois group of the splitting field of $\varphi(x)$ over $K$. 
Let $k$ be any field and let $s_0,s_1,\ldots, s_{d-1},u_0,u_1,\ldots,u_d,x$ be independent indeterminants over $k$. The polynomial $$\fG(x)=x^d+s_{d-1}x^{d-1}+\dots+s_0$$ is the \textit{generic monic polynomial} of degree $d$ over $k$. The rational function $$\Phi(x)=\frac{x^d+s_{d-1}x^{d-1}+\ldots+s_0}{u_dx^d+u_{d-1}x^{k-1}+\ldots+u_0}$$ is the \textit{generic rational function} of degree $d$ over $k$. 

Define the field $k(\textbf{s})$ by $k(\textbf{s}):=k(s_0,s_1,\ldots, s_{d-1})$ and similarly define $k(\textbf{s,u}):=k(s_0,s_1,\ldots, s_{d-1},u_0,u_1,\ldots,u_d)$. In \cite{odoni}, Odoni shows that if $\ch k =0$, then $\Gal(\fG^n(x)/k(\textbf{s}))$ is isomorphic to $[S_d]^n$, the $n$-th wreath power of the symmetric group $S_d$. This group can be thought of as the group of automorphisms of the $d$-ary rooted tree up to the $n$-th level. Wreath products are defined in Section \ref{wreath}. 

Some of the arguments used in \cite{odoni} do not extend to positive characteristic, such as those dealing with the theory of monodromy groups on compact Riemann surfaces and branch points of algebraic functions over $\bC$. Here, we instead use algebraic and Galois theoretic arguments to show the following.

\begin{theorem} \label{generic} For any field $k$, $d > 1$, and $n\in \bN$, $\Gal(\Phi^n(x)/k(\textbf{s,u})) \cong [S_d]^n$ and if $(d,p)\neq (2,2)$, then
	$\Gal(\fG^n(x)/k(\textbf{s})) \cong [S_d]^n$.
\end{theorem}

It can be easily shown that, $\Gal(\fG^n(x)/k(\textbf{s}))$ and $\Gal(\Phi^n(x)/k(\textbf{s,u}))$ must be contained in $[S_d]^n$. The majority of the work in this paper involves showing that these Galois groups contain $[S_d]^n$ as well. 

First, note that by passing to an algebraic closure of $k$ these Galois groups can only decrease in size. So we may replace $k$ with an algebraic closure of $k$ and prove the result in this case. Let $f(x)\in k[x]$ be any polynomial with degree $d$, and let $t$ be transcendental over $k$. If we define $g(x):=f(x+t)-t$, then $g^n(x)=f^n(x+t)-t$ for any $n \in \bN$. So it follows that $\Gal(g^n(x)/k(t))\cong\Gal(f^n(x)-t/k(t))$ for any $n \in \bN$.  Then, since $g^n(x)$ is a specialization of both $\fG^n(x)$ and $\Phi^n(x)$, for any $n$, and Galois groups cannot increase under specializations, it will suffice to show that there exists some $f(x)\in k[x]$ such that $\Gal(f^n(x)-t/k(t))\cong[S_d]^n$. 

In Theorem \ref{sufficient}, we give sufficient conditions on rational functions $\varphi(x)\in k(x)$ to ensure $\Gal(\varphi^n(x)-t/k(t))\cong [S_d]^n$. Then in Theorem \ref{galois group}, we show that in fact ``most'' polynomials in $k[x]$ satisfy these conditions. 

The proof of Theorem \ref{sufficient} is by induction on $n$. The main tool used in the inductive step is ``disjoint ramification'' of primes, that is, we show that in each subextension of the splitting field of $\varphi^n(x)-t$ formed by adjoining $\alpha$ and $\varphi^{-1}(\alpha)$ where $\alpha\in\varphi^{-(n-1)}(t)$, there is a prime that ramifies which ramifies in no other such subextension. The arguments here are similar to the arguments given in \cite{JKMT}. 

We give some preliminary results in Section \ref{preliminaries}. We prove Theorem \ref{generic} in Section \ref{Generic}. In Section \ref{2,2}, we handle the case $\ch k = d= 2$.  In this case it is worth noting that rational functions behave as in all the other cases whereas the results for polynomials are markedly different. The difference follows from the fact that these polynomials are always post-critically finite. 

Finally, in Section \ref{applications} we give two applications using these results along with appropriate versions of the Chebotarev Density Theorem. We first use Theorem \ref{generic} to extend Odoni's application on primes dividing orbits (\cite{odoni}, Lemma 9.1) to global fields in any characteristic, where by global field we mean a number field or a function field of an algebraic curve over a finite field. 

In the second application, we consider factorizations of iterates of polynomials. Let $\pi=(1)^{r_1}\dots(m)^{r_m}$ be a cycle pattern in $S_m$. We say that a squarefree polynomial $f(x)$ of degree $m$ has cycle pattern $\pi$ if $f(x)$ has exactly $r_i$ irreducible factors of degree $i$ for all $1\leq i\leq m$. If $\pi$ is a \textit{cycle pattern} in $S_{d^n}$, let $A(q,b,d,n,\pi)$ be the set of all $f(x)\in \F_q[x]$ such that $\deg f(x)=d$, $f(x)$ has leading coefficient $b$, $f^n(x)$ is squarefree, and $f^n(x)$ has cycle pattern $\pi$.
We show there is an $M=M(d,n)>0$ in $\bR$ and $q_0(d,n)\geq 2$  in $\bN$ such that 
\[|\#A(q,b,d,n,\pi)-q^d\rho(\pi)|\leq Mq^{d-\frac{1}{2}}\]
whenever $(d,\ch \F_q)\neq (2,2)$, $d\geq 2$, $n\geq 1$, $b
\in \F_q^*$,  and $q\geq q_0$. Here $\rho(\pi)$ is the proportion of elements of $[S_d]^n$ with cycle pattern $\pi$.

\vskip2mm
\noindent {\it Acknowledgments.} This paper is dedicated to R.W.K. Odoni for his work on this project and many others that led to the development of a rich subject area. The author would also like to thank Thomas J. Tucker for suggesting this project and for many useful conversations, and Michael Zieve for finding and providing her with the preprint containing R.W.K. Odoni's unpublished work. Thanks are due to S.D. Cohen, with whom R.W.K Odoni shared many useful discussions about this project. This work was partially supported by NSF grant DMS-1200749. 


\section{Preliminaries} \label{preliminaries}

\subsection{Wreath Products}\label{wreath} We first define the wreath product of groups acting on finite sets. For a more detailed description see \cite{Nek}.

\begin{definition} Let $G$ and $H$ be groups acting on the finite sets $\{\alpha_1,\ldots,\alpha_d\}$ and $\{\beta_1,\ldots,\beta_\ell\}$ respectively. The set $\{(\pi,\tau_1,\ldots,\tau_d)|\pi\in G, \tau_1,\ldots,\tau_d \in H\}$ forms a group called the \emph{wreath product of $G$ by $H$}, denoted $G[H]$. $G[H]$ acts on the set $\{\alpha_1,\ldots,\alpha_d\}\times\{\beta_1,\ldots,\beta_\ell\}$ by $(\alpha_i,\beta_r)\mapsto(\alpha_{\pi(i)},\beta_{\tau_i(r)})$. 
\end{definition}

\begin{lemma}[\cite{odoni}, Lemma 4.1] Let $\varphi(x),\psi(x)$ be rational functions with coefficients in a field $K$ with $\deg(\varphi)=d$ and $ \deg(\psi)=\ell$, and $d,\ell \geq 1$, such that $\varphi(\psi(x))$ has $d\ell$ distinct roots in $\bar{K}$. Let $G=\Gal(\varphi(x)/K)$. Then $\Gal(\varphi(\psi(x))/K)$ is isomorphic to a subgroup of $G[S_\ell]$.
\end{lemma} 

\begin{proof} Let $\{\alpha_1,\ldots,\alpha_d\}$ be the roots of $\varphi(x)$ then the roots of $\varphi(\psi(x))$ are the roots of $(\psi(x)-\alpha_i)$ for $i=1,\ldots, d$. So we can write the set of roots of $\varphi(\psi(x))$ as $\{\beta_{i,r}| i=1,\ldots,d, r=1,\ldots,\ell\}$ where $\{\beta_{i,r}|r=1,\ldots,\ell\}$ is the set of zeros of $\psi(x)-\alpha_i$. Let $\sigma \in \Gal(\varphi(\psi(x))/K)$.  Let $F$ be the splitting field of $\varphi(x)$ over $K$, then $\sigma$ induces a permutation $\pi:=\sigma|_F$ on $\{\alpha_1,\ldots,\alpha_d\}$, that is, $\pi \in G$. We can think of $\pi$ as a permutation on the indicies $\{1,\dots,d\}$ defined by $\alpha_{\pi(i)}:=\pi(\alpha_i)$. Now fix $i$, and note that since $\psi(\sigma(\beta_{i,r}))=\sigma(\psi(\beta_{i,r}))=\sigma(\alpha_i)=\pi(\alpha_i)=\alpha_{\pi(i)}$, we must have $\sigma(\beta_{i,r})=\beta_{\pi(i),s}$ for some $s$. This defines a map $r\mapsto s$ which is a permutation of $\{1,\ldots,\ell\}$, we call this map $\tau_i$. Hence, the map $\sigma$ is given by $\sigma(\beta_{i,r})=\beta_{\pi(i),\tau_{i}(r)}$ for $\pi\in G$ and $\tau_i\in S_\ell$. Thus, we can define a map $\Gal(\varphi(\psi(x))/K) \lra G[S_\ell]$ by $\sigma\mapsto (\pi,\tau_1,\ldots,\tau_d)$ which is easily shown to be an injective homomorphism.
\end{proof}

We define the \textit{$n$-th wreath power of a group $G$} recursively by $[G]^1=G$ and $[G]^n=[G]^{n-1}[G]$.

\begin{corollary} \label{iterated wreath} If $\varphi(x)$ is a rational function in $K(x)$ with degree $d$ and $\alpha\in K$ such that $\varphi^n(x)-\alpha$ has $d^n$ distinct zeros in $\bar{K}$, then $\Gal(\varphi^n(x)-\alpha/K)$ can be embedded in $[S_d]^n$. 
\end{corollary}

The group  $[S_d]^n$ has a nice interpretation as the automorphism group of the $d$-ary rooted tree up to the $n$-th level \cite{Nek}, \cite{BJ}. Suppose $\varphi^n(x)-\alpha$ has $d^n$ distinct roots for each $n\geq 1$. Since the image of any root of $\varphi^n(x)-\alpha$ under $\varphi$ is a root of $\varphi^{n-1}(x)-\alpha$, we can define a tree structure on the roots of $\varphi^n(x)-\alpha$ as follows. If $\beta\in\varphi^{-n}(\alpha)$ then $\beta$ lies in the $n$-th level of the tree. If $\varphi(\beta)=\gamma$, then $\beta$ lies above $\gamma$ in the tree, that is, there is a branch connecting $\beta$ to $\gamma$. The diagram is shown in Figure \ref{tree}.

	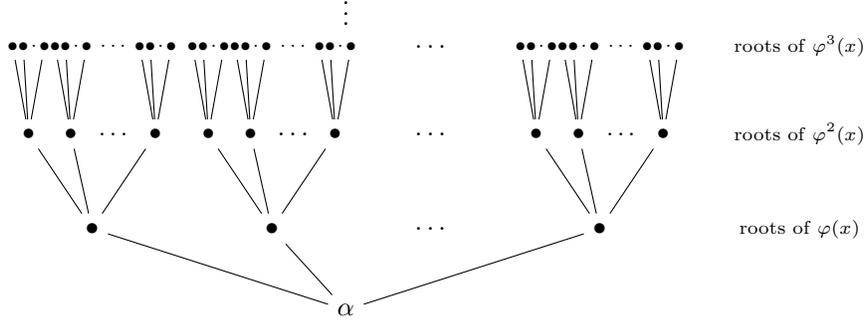
\begin{figure}
		\begin{tikzpicture}[scale=.9]
		
		\path 
		(5,1)node(0){$\alpha$}
		(10/8,2.2)node(1){$\bullet$}
		(30/8+20/128,2.2)node(2){$\bullet$}
		(50/8,2.2)node(m){$\ldots$}
		(70/8,2.2)node(3){$\bullet$}
		(11.7,2.2)node{\scriptsize{roots of $\varphi(x)$}};
		
		\draw(1)to(0);
		\draw(2)to(0);
		\draw(3)to(0);
		
		\path(10/32,3.6)node(11){\small$\bullet$}
		(30/32,3.6)node(12){\small$\bullet$}
		(50/32,3.6)node(1m){\small$\ldots$}
		(70/32,3.6)node(13){\small$\bullet$}
		
		(90/32+20/128,3.6)node(21){\small$\bullet$}
		(110/32+20/128,3.6)node(22){\small$\bullet$}
		(130/32+20/128,3.6)node(2m){$\ldots$}
		(150/32+20/128,3.6)node(23){\small$\bullet$}
		
		(200/32,3.6)node(mm){$\ldots$}
		
		(250/32,3.6)node(31){\small$\bullet$}
		(270/32,3.6)node(32){\small$\bullet$}
		(290/32,3.6)node(3m){\small$\ldots$}
		(310/32,3.6)node(33){\small$\bullet$}
		(11.7,3.6)node{\scriptsize{roots of $\varphi^2(x)$}};
		
		\draw(11)to(1)to(12);
		\draw(13)to(1);
		\draw(21)to(2)to(22);
		\draw(23)to(2);
		\draw(31)to(3)to(32);
		\draw(33)to(3);
		
		\path(10/128,4.9)node(111){\scriptsize$\bullet$}
		(30/128,4.9)node(112){\scriptsize$\bullet$}
		(50/128,4.9)node(11m){\scriptsize$\ldots$}
		(70/128,4.9)node(113){\scriptsize$\bullet$}
		
		(90/128,4.9)node(121){\scriptsize$\bullet$}
		(110/128,4.9)node(122){\scriptsize$\bullet$}
		(130/128,4.9)node(12m){\scriptsize$\ldots$}
		(150/128,4.9)node(123){\scriptsize$\bullet$}
		
		(200/128,4.9)node(1mm){\scriptsize$\ldots$}
		
		(250/128,4.9)node(131){\scriptsize$\bullet$}
		(270/128,4.9)node(132){\scriptsize$\bullet$}
		(290/128,4.9)node(13m){\scriptsize$\ldots$}
		(310/128,4.9)node(133){\scriptsize$\bullet$}

		(330/128+20/128,4.9)node(211){\scriptsize$\bullet$}
		(350/128+20/128,4.9)node(212){\scriptsize$\bullet$}
		(370/128+20/128,4.9)node(21m){\scriptsize$\ldots$}
		(390/128+20/128,4.9)node(213){\scriptsize$\bullet$}
		
		(410/128+20/128,4.9)node(221){\scriptsize$\bullet$}
		(430/128+20/128,4.9)node(222){\scriptsize$\bullet$}
		(450/128+20/128,4.9)node(22m){\scriptsize$\ldots$}
		(470/128+20/128,4.9)node(223){\scriptsize$\bullet$}
		
		(520/128+20/128,4.9)node(2mm){\scriptsize$\ldots$}
		
		(570/128+20/128,4.9)node(231){\scriptsize$\bullet$}
		(590/128+20/128,4.9)node(232){\scriptsize$\bullet$}
		(610/128+20/128,4.9)node(23m){\scriptsize$\ldots$}
		(630/128+20/128,4.9)node(233){\scriptsize$\bullet$}   (11.7,4.9)node{\scriptsize{roots of $\varphi^3(x)$}}
		
		(5,5.5)node{\vdots}
		
		(200/32,4.9)node(mmm){$\ldots$}

		(970/128,4.9)node(311){\scriptsize$\bullet$}
		(990/128,4.9)node(312){\scriptsize$\bullet$}
		(1010/128,4.9)node(31m){\scriptsize$\ldots$}
		(1030/128,4.9)node(313){\scriptsize$\bullet$}
		
		(1050/128,4.9)node(321){\scriptsize$\bullet$}
		(1070/128,4.9)node(322){\scriptsize$\bullet$}
		(1090/128,4.9)node(32m){\scriptsize$\ldots$}
		(1110/128,4.9)node(323){\scriptsize$\bullet$}
		
		(1160/128,4.9)node(3mm){\scriptsize$\ldots$}
		
		(1210/128,4.9)node(331){\scriptsize$\bullet$}
		(1230/128,4.9)node(332){\scriptsize$\bullet$}
		(1250/128,4.9)node(33m){\scriptsize$\ldots$}
		(1270/128,4.9)node(333){\scriptsize$\bullet$};
		
		\draw(111)to(11);
		\draw(112)to(11)to(113);
		\draw(121)to(12);
		\draw(122)to(12)to(123);
		\draw(131)to(13);
		\draw(132)to(13)to(133);
		\draw(211)to(21);
		\draw(212)to(21)to(213);
		\draw(221)to(22);
		\draw(222)to(22)to(223);
		\draw(231)to(23);
		\draw(232)to(23)to(233);
		\draw(311)to(31);
		\draw(313)to(31);
		\draw(312)to(31);
		\draw(321)to(32);
		\draw(322)to(32)to(323);
		\draw(331)to(33);
		\draw(332)to(33)to(333);
		\end{tikzpicture}
		\caption{Tree diagram for the roots of $\varphi^n(x)$.}
		\label{tree} 
	\end{figure}

The group $\Gal(\varphi^n(x)-\alpha/K)$ acts on the tree up to the $n$-th level by permuting the branches so $\Gal(\varphi^n(x)-\alpha/K)$ is isomorphic to a subgroup of $\Aut(T_{d,n})$, the automorphism group of the tree up to the $n$-th level.

\subsection{Discriminants and Ramification}

Let $M/K$ be a finite Galois extension. If $\fp$ is a prime of $K$ and $\fq$ is any prime of $M$ extending $\fp$, we define $e(\fq|\fp)$ to be the inertia degree of $\fq$ over $\fp$ and $f(\fq|\fp)$ to be the residue degree of $\fq$ over $\fp$. The next result is useful in determining the structure of inertia groups, similar results can be found in \cite{GTZ}, \cite{vanderWaerden}.

\begin{lemma} \label{inertia group}
	Let $M/K$ be a finite Galois extension with Galois group $G$. Let $H$ be a subgroup of $G$ and $L = M^H$ be the corresponding intermediate field. Let $\fq$ be a prime of $M$ and $\fp := \fq \cap K$. Let $X$ be the transitive $G$-set $G/H$. Then there is a bijection between the set of orbits of $X$ under the action of $D(\fq|\fp)$, the decomposition group of $\fq$ over $\fp$, and the set of extensions $\fP$ of $\fp$ to $L$ with the property: If $\fP$ corresponds to $Y$ then the length of $Y$ is $e(\fP|\fp)f(\fP|\fp)$ and $Y$ is the disjoint union of $f(\fP|\fp)$ orbits of length $e(\fP|\fp)$ under the action of $I(\fq|\fp)$, the inertia group of $\fq$ over $\fp$.  
\end{lemma}

\begin{proof}  For $\tau \in G$ we will show that the length of the orbit of the coset $H\tau$ under the action of $D(\fq|\fp)$ is $e(\fP|\fp)f(\fP|\fp)$, where $\fP=\tau(\fq) \cap L$. Let $Y$ be the orbit of $H\tau$ and $\Stab_{D(\fq|\fp)} (H\tau)$ be the stabilizer of $H\tau$ under the action of $D(\fq|\fp)$. Then,
	
	\begin{align*}
	\Stab_{D(\fq|\fp)} (H\tau) & = \{ \gamma \in D(\fq|\fp) | H\tau\gamma = H\tau\} 
	=\{\gamma \in D(\fq|\fp)|\tau\gamma\tau^{-1} \in H\} \\
	& =H\cap \tau D(\fq|\fp) \tau^{-1} 
	= H \cap D(\tau(\fq)|\fp) 
	= D(\tau(\fq)|\fP),
	\end{align*}

	\noindent where $D(\tau(\fq)|\fP)$ is the decomposition group of $\tau(\fq)$ over $\fP$. So, the Orbit/Stabilizer Theorem implies 
	
	\begin{align*}
	\# Y & =\frac{\#D(\fq|\fp)}{\#\Stab_{D(\fq|\fp)} (H\tau)}
	= \frac{\#D(\fq|\fp)}{\#D(\tau(\fq)|\fP)} 
	= \frac{\#D(\fq|\fp)}{\#D(\fq|\fP)}
	= e(\fP|\fp)f(\fP|\fp). 
	\end{align*}

	Now, we must show that this correspondence is well-defined and bijective. Suppose $H\tau$ and $H\sigma$ are in the same orbit under the action of $D(\fq|\fp)$, then $\exists \gamma \in D(\fq|\fp)$ such that $H\tau\gamma=H\sigma$ which implies $ \tau\gamma\sigma^{-1} \in H$. So $\sigma(\fq)\cap L= \tau\gamma\sigma^{-1}(\sigma(\fq)\cap L)=\tau(\fq)\cap L$ and the map $Y \mapsto \tau(q)\cap L$ is well-defined. Clearly the map is surjective, since $G$ permutes the primes of $M$ lying above $\fp$ transitively. To see that this map is one-to-one suppose $\tau(\fq)\cap L=\sigma(\fq)\cap L= \fP$. Then $\tau(\fq), \sigma(\fq)$ both lie above $\fP$, and since $H$ acts transitively on the primes of $M$ lying above $\fP$, $\exists \gamma \in H$ such that $\gamma\tau(\fq)=\sigma(\fq)$. Then, $\sigma^{-1}\gamma\tau(\fq)=\fq$ so $\sigma^{-1}\gamma\tau \in D(\fq|\fp)$. Since $H\sigma(\sigma^{-1}\gamma\tau) = H\tau$, this shows that $H\sigma$ and $H\tau$ are in the same orbit under the action of $D(\fq|\fp)$.
	
	It remains to show that $Y$ is the disjoint union of $f(\fP|\fp)$ orbits of length $e(\fP|\fp)$ under the action of $I(\fq|\fp)$. Let $Z$ be the orbit of $H\tau$ under $I(\fq|\fp)$, it suffices to show that $\#Z$ is $e(\fP|\fp)$. Let $\Stab_{I(\fq|\fp)} (H\tau)$ be the stabilizer of $H\tau$ under the action of $I(\fq|\fp)$. Then, arguing as before we see,
	
	\[
	\Stab_{I(\fq|\fp)} (H\tau) = \{\gamma \in I(\fq|\fp)|H\tau\gamma=H\tau\} \\
	= I(\tau(\fq)|\fP) 
	\]

	\noindent where $I(\tau(\fq)|\fP)$ is the inertia group of $\tau(\fq)$ over $\fP$. Using the Orbit/Stabilizer Theorem again,

	\[
	\# Z = \frac{\#I(\fq|\fp)}{\#\Stab_{I(\fq|\fp)} (H\tau)}\\
	=e(\fP|\fp).  
	\]
\end{proof}

\begin{remark} The set $G/H$ is the set of $K$ homomorphisms  of $L$ into $M$. In the case $L \cong K(\theta)$ where $\theta$ is a root of some $f \in K[x]$, this corresponds to the set of zeros of $f$ in $M$, so there is a one-to-one correspondence between the set of orbits of the roots of $f(x)$ under the action of the decomposition group $D(\fq|\fp)$ and the set of extension of $\fp$ to $L$ with the property from Lemma \ref{inertia group}. 
\end{remark}

Let $A$ be a Dedekind domain, $K$  the field of fractions of $A$,  $L$  a separable extension of
$K$, and $B$ the integral closure of $A$ in $L$. It is a
standard result that any prime of $A$ that ramifies in $B$ must contain $\Delta(B/A)$, the 
discriminant ideal of the extension $B/A$.  The following two results on discriminants are standard, see  \cite{Jan} or \cite{Lang}, for example.

\begin{lemma}\label{ramification and discriminant}  Let $\fp \subseteq A$ be a prime, $\fp B=\prod{\fq_i^{e_i}}$, and $f_i=f(\fq_i|\fp)$ the residue degree, then the power of $\fp$ in $\Delta(B/A)$ is greater than or equal to $\sum{(e_i-1)f_i}$ with equality if and only if $\ch K$ does not divide $e_i$ for any $i$. 
\end{lemma}

For computational purposes it is often easier to work with polynomial discriminants which we will do here.

\begin{lemma}\label{discriminant of extension}  Let $P(x)$ be an irreducible polynomial in $A[x]$, let $\theta$ be a root of $P(x)$, and let $L=K(\theta)$, if $B=A[\theta]$, that is, if $A[\theta]$ integrally closed in $L$, then $\Delta(B/A)=(\Delta(P(x)))$, where $\Delta(P(x))$ is the usual polynomial discriminant of $P(x)$ and $(\Delta(P(x)))$ is the ideal generated by $\Delta(P(x))$.
\end{lemma}

Thus, if $B=A[\theta]$ then the only primes of $A$ ramifying in $B$ must divide $\Delta(P(x))$, and furthermore, if $\fp$ ramifies in $B$ then $v_{\fp}(\Delta(P(x))=v_{\fp}(\Delta(B/A))$. In the next two corollaries we assume this is the case and let $M$ be the splitting field of $P(x)$ over $K$.

\begin{corollary}\label{singletransposition}  If $\fp||\Delta(P(x))$  in $A$, then for any prime $\fq$ of $M$ lying over $\fp$, the action of the inertia group $I(\fq|\fp)$ on the roots of $P(x)$ consists of a single transposition.
\end{corollary}

\begin{proof}
	Let $\{\alpha_1,\ldots,\alpha_d\}$ be the roots of $P(x)$ in $M$ and $L=K(\alpha_i)$ for some $i$. Since $\fp||\Delta(P(x))$, Lemma \ref{discriminant of extension} implies $\fp||\Delta(B/A)$, where $B$ is the integral closure of $A$ in $L$. Then by Lemma \ref{ramification and discriminant}, $\fp B =\fP_1^2\fP_2\ldots \fP_m$ where $f(\fP_1|\fp)=1$ for some primes $\fP_1,\ldots,\fP_m$ in $B$. If $\fq$ is a prime of $M$ lying over $\fp$, then by Lemma \ref{inertia group}, the action of $I(\fq|\fp)$ on $\{\alpha_1,\ldots,\alpha_d\}$ consists of a single transposition.
\end{proof}

\begin{corollary}\label{transposition or three cycle} If $\ch(K)=2$ and $\fp^2||\Delta(P(x))$ in $A$, then for any prime $\fq$ of $M$ lying over $\fp$, the action of $I(\fq|\fp)$ on the roots of $P(x)$  consists of a single transposition or a single three cycle. 
\end{corollary}

\begin{proof}
	With notation as in the proof of Corollary \ref{singletransposition}, Lemma \ref{ramification and discriminant} implies $\fp B =\fP_1^2\fP_2\ldots \fP_m$ where $f(\fP_1|\fp)=1$, or  $\fp B =\fP_1^3\fP_2\ldots \fP_m$ where \newline{} $f(\fP_1|\fp)=1$, for some primes $\fP_1,\ldots,\fP_m$ in $B$. Then, if $\fq$ is a prime of $M$ lying over $\fp$, by Lemma \ref{inertia group}, the action of $I(\fq|\fp)$ on $\{\alpha_1,\ldots,\alpha_d\}$ consists of a single transposition, or a single three cycle, respectively.
\end{proof}

We often work with splitting fields of rational functions, in this case we need the following useful result of Cullinan and Hajir \cite{CH}. Let $t$ be transcendental over a field $k$ and $\psi(x)=p(x)/q(x)$ be a rational function with coefficients in $k$ and $p(x)$, $q(x) \in k[x]$, then the splitting field of $\psi(x)-t$ over $k(t)$ is the splitting field of the polynomial $p(x)-tq(x)$ over $k(t)$.

\begin{lemma}[\cite{CH}, Proposition 1] \label{discriminant formula}  We have 
	\begin{align*}
	\Delta(p(x)-tq(x)) &= C \res(p'(x)q(x)-p(x)q'(x), p(x)-tq(x))\\
	&= C' \prod_{a\in\psi_\mathfrak{c}}(\psi (a)-t)^{e(a|\psi(a))-1}
	\end{align*}
	where $C,C'\in k$ are constants, $\psi_\mathfrak{c}=\{a \in \bar{k}:\psi
	'(a)=0\}$, and $e(a|\psi(a))$ is the ramification index of $a$ over
	$\psi(a)$.  
\end{lemma}

Thus, we see that any prime $\fp$ of $k[t]$ that ramifies in
a splitting field for $p(x) - t q(x)$  must divide $\prod_{a\in\psi_\mathfrak{c}}(\psi (a)-t)^{e(a|\psi(a))-1}$.  We will use the notation \[\Delta(\psi(x)-t):=\prod_{a\in\psi_\mathfrak{c}}(\psi (a)-t)^{e(a|\psi(a))-1}.\]

The following result is a standard consequence of the Riemman Hurwitz formula (see \cite{stichtenoth}, for example).
\begin{lemma}\label{riemann hurwitz}  For any field $k$, $k(t)$ has no finite separable extensions with constant field $k$ of degree $d\geq 2$ which are unramified over all $\fp \in \bP_{k(t)}\setminus \{\fp_\infty\}$ and tamely ramified at $\fp_\infty$, here $\bP_{k(t)}$ denotes the set of primes of $k(t)$.
\end{lemma}

\begin{proof} Let $F$ be any finite extension of $k(t)$ with field of constants $k$ and let $d =[F:k(t)]$. For the prime $\fp_\infty$ of $k(t)$, $\sum_{\fq|\fp_\infty} (e(\fq|\fp_\infty)-1) \deg \fq \leq d-1$ where the sum ranges over all $\fq$ extending $\fp_\infty$. Let $g$ be the genus of $F$. From the Riemann Hurwitz formula we have 
	\begin{align*}
	2g-2&= -2d +\sum_{\fp\in \bP_{k(t)}} \sum_{\fp'|\fp} (e(\fp'|\fp)-1) \deg \fp'  \\
	2g-2&\leq -2d +\sum_{\substack{\fp\in \bP_{k(t)}\\ \fp\neq\fp_\infty}} \sum_{\fp'|\fp} (e(\fp'|\fp)-1) \deg \fp' + d-1 ,\\
	\end{align*}
	where the second sum is taken over all $\fp'$ extending $\fp$ in $F$. Then since $g\geq0$, 
	\[
	d-1 \leq \sum_{\substack{\fp\in \bP_{k(t)}\\ \fp\neq\fp_\infty}} \sum_{\fp'|\fp} (e(\fp'|\fp)-1) \deg \fp'.\]
	Since $d\geq 2$, some prime in $\bP_{k(t)}\setminus \{\fp_\infty\}$ must ramify in $F$. 
\end{proof}

\subsection{Results on subgroups of $S_d$}

Let $F$ be any field, we say a polynomial $f(x) \in F[x]$
is \textit{indecomposable} if $f(x)$ cannot be written 
as $f(x)=g(h(x))$ for $g,h \in F[x]$ with $\deg g, \deg h >1$. A group $G$ acting on a set $S$ is said to be \textit{primitive} if it acts transitively and preserves no nontrivial partition of $S$. The following is a result of Fried \cite{Fried} as referenced in \cite{Cohen}.

\begin{lemma}[\cite{Cohen}, Lemma 3.1] \label{indecomposable}  A separable polynomial $f(x)$ over a field $F$ is indecomposable if and only if the Galois group $G$ of $f(x)-t$ over $F(t)$ is primitive on the roots of $f(x)-t$. 
\end{lemma}

\begin{proof} Let $\{a_1,\dots,a_d\}$ be the roots of $f(x)-t$. Since $f(x)-t$ is irreducible over $F(t)$, $G$ must be transitive. Suppose $G$ is imprimitive. Then there is some nontrivial partition of $\{a_1,\dots,a_d\}$ into disjoint subsets $S_1,\dots,S_n$ preserved by $G$. Let $S=S_i$ be one of these subsets with $\#S_i>1$. If $a\in S$ then $\Stab_G(a) \subsetneq \{\sigma \in G|\sigma(S)=S\}$. So $\Stab_G(a)$ is not a maximal subgroup of $G$. Hence, there is a field strictly between $F(t)$ and $F(a)$, which by Luroth's Theorem, must be of the form $F(u)$. Thus, $u=h(a)$ and $t=g(u)$ for (non-linear) rational functions $g,h$ with coefficients in $F$. Then since $f(a)=g(h(a))$, we can find (non-linear) polynomials $g_1,h_1$ with coefficients in $F$ such that $f(x)=g_1(h_1(x))$. Thus, $f$ is decomposable over $F$. 
	
	Conversely, if $f$ is decomposable then we can write $f(x)=g(h(x))$ for non-linear $g(x), h(x) \in F[x]$. Then the relation $a\sim b$ if $h(a)=h(b)$ gives a nontrivial partition of $\{a_1,\dots,a_d\}$ preserved by $G$.
\end{proof}

The next two results are standard and are provided here for completeness.

\begin{lemma}\label{primitive} If $G$ is a primitive subgroup of $S_d$ that contains a transposition then $G=S_d$. 
\end{lemma}

\begin{proof} Define a relation on $\{1,\dots,d\}$ by $i\sim j$ if either $i=j$ or $G$ contains the transposition $(i j)$. This is clearly a $G$-invariant equivalence relation. Since $G$ contains a transposition, there are fewer than $d$ equivalence classes. Then since $G$ is primitive, there must be only one equivalence class. So $G$ contains all the transpositions which implies $G=S_d$.
\end{proof}

\begin{lemma}\label{transitive subgroup} If $G$ is a transitive subgroup of $S_d$ that is generated by transpositions then $G=S_d$.
\end{lemma}

\begin{proof} Let $\cS$ be the set of all $k$ such that there exists some subgroup $H$ of $G$ isomorphic to $S_k$. $\cS$ is nonempty since there are subgroups of $G$ which are isomorphic to $S_1$ and $S_2$. Let $m\in \cS$ be maximal. Suppose $m \neq d$, after renumbering elements of $\{1,2,\dots,d\}$ we can assume that $H$ acts on $\{1,2,\dots,m\}$. Now, since $G$ is transitive and generated by transpositions, there is some $(i j) \in G$ such that $i\in \{1,\dots,m\}$ and $j > m$. But then the subgroup of $G$ generated by $H\cup\{(i j)\}$ is isomorphic to $S_{m+1}$, contradicting the maximality of $m$.
\end{proof}

\subsection{The Zariski Topology on $\mathcal{P}_d(k)$ and $\operatorname{Rat}_d(k)$}

Let $k$ be an algebraically closed field. 
Given a point $(a_0,\dots,a_d)$ in $\bA^{d+1}(k)$ with $a_d \neq 0$, $a_dx^d+ \dots+ a_0$ is a polynomial of degree $d$ in $k[x]$. 
We denote the set of all such $(a_0,\dots,a_d)$ by $\cP_d(k)$ and give $\cP_d(k)$ the subspace topology inherited from the Zariski toplogy on $\bA^{d+1}(k)$. 

Similarly, given a point $(a_0,\dots, a_d, b_0,
\dots, b_d)$ in $\bA^{2d+2}(k)$, we set  $p = a_d x^d + \dots + a_0$,
$q = b_d x^d + \dots + b_0$, and $\varphi = p/q$.  If the resultant of $p$ and $q$ is
nonzero and either $a_d$ or $b_d$ is nonzero, then $\varphi$ is a rational function of
degree $d$ in $k(x)$.  We denote the set of such $(a_0,\dots, a_d, b_0,
\dots, b_d)$ by
$\Rat_d(k)$ and give $\Rat_d(k)$ the subspace topology inherited from the Zariski topology on $\bA^{2d+2}(k)$.


\section{Galois Groups of $\varphi^n(x)-t$}\label{H}

In this section, let $k$ be a field with characteristic $p$ (where $p$ is allowed to be $0$), let  $x,t$ be algebraically independent variables over $k$, and let $\varphi(x) \in k(x)$ be a rational function with degree $d>1$. Then, for $n\in \bN$, $\varphi^n(x)-t$ is irreducible, and if $\frac{d}{dx} \varphi(x)\neq 0$ then $\varphi^n(x)-t$ is $x$-separable, since $\frac{d}{dx} (\varphi^n(x)-t)=\frac{d}{dx} \varphi^n(x) \neq 0$ by induction on $n$. For fixed $N\in\bN$, we give conditions on $\varphi(x)$ that ensure  $\Gal(\varphi^N(x)-t/k(t)) \cong [S_d]^N$.  Then we show that these conditions are not too restrictive as long as $(d,p) \neq (2,2)$ by showing that when $k$ is algebraically closed the set of all $f(x)\in k[x]$ of degree $d$ with the property $\Gal(f^N(x)-t/k(t))\cong [S_d]^N$ contains a nonempty Zariski-open subset of $\cP_d(k)$.

\begin{theorem}\label{sufficient} Let $\varphi(x) \in k(x)$ with $\Gal(\varphi(x)-t/k(t))\cong S_d$. If $\ch k \neq 2$, suppose $\varphi$ has some critical point $a \in k$ with multiplicity one such that $\varphi^n(a)\neq \varphi^m(b)$ for all $m\leq n\leq N$, unless $m=n$ and $b=a$, and if $\ch k =2$, suppose $\varphi$ has some critical point $a$ such that $\varphi^n(a)\neq \varphi^m(b)$ for all $m\leq n\leq N$, unless $m=n$ and $b=a$ and $I(\fq|\fp)$ consists of a single transposition for any prime $\fq$ lying above $\fp=k(t)\cap(\varphi(a)-t)$ in the splitting field of $\varphi(x)-t$ over $k(t)$. Then $\Gal(\varphi^N(x)-t/k(t))\cong [S_d]^N$.
\end{theorem}

The following proposition follows almost immediately from Theorem \ref{sufficient}.

\begin{proposition} Suppose $\ch k \neq 2$ and $d \nmid \ch k$. Then if $\varphi(x)$ is any rational function with degree $d$ such that each of the critical points in $\overline{k}$ have multiplicity one and there is some critical point $a$ such that $\varphi^n(a)\neq \varphi^m(b)$ for any critical point $b\neq a$ and any $m\leq n$, we have $\Gal(\varphi^N(x)-t/k(t))\cong [S_d]^N$ for all $N$.
\end{proposition}

\begin{proof} By Theorem \ref{sufficient} we only need to show that $G=\Gal(\varphi(x)-t/k(t))\cong S_d$. Let $K_1$ be the splitting field of $\varphi(x)-t$ over $k(t)$. By Lemma \ref{discriminant formula} the discriminant of $\varphi(x)-t$ is squarefree so Corollary \ref{singletransposition} implies that for any ramified prime $\fq$ lying over a prime $\fp$, $I(\fq|\fp)$ consists of a single transposition. Let $I \subseteq G$ be the subgroup generated by $\{I(\fq|\fp): \fq|\fp, \fq\in \bP_{K_1}, \text{ and } \fp\in \bP_{k(t)}\setminus\{\fp_\infty\}\}$. Then $K_1^I$ is unramified over all primes of $k[t]$, and by Lemma \ref{riemann hurwitz}, $K_1^I=k(t)$. Thus, $G=I$. So $G$ is a transitive subgroup of $S_d$ generated by transpositions and Lemma \ref{transitive subgroup} implies $G \cong S_d$.
\end{proof}

Before we prove Theorem \ref{sufficient}, which gives conditions ensuring $\Gal(\varphi^N(x)-t/k(t))$ is isomorphic to $[S_d]^N$, we fix some notation and prove a lemma. For $n<N$, let $K_n$ be the splitting field of $\varphi^n(x)-t$ over $k(t)$, $\alpha_1,\dots,\alpha_{d^n}$ the roots of $\varphi^n(x)-t$, $M_i$ the splitting field of $\varphi(x)-\alpha_i$ over $k(\alpha_i)=k(\alpha_i,t)$, and $\widehat{M_i}:=K_n\prod_{j\neq i}M_j$. 

In order to work with discriminants as in Lemma \ref{discriminant formula} we need to make a few reductions.  First note that
for any extension $k'$ of $k$, we have $\Gal(K_N\cdot k' /
k'(t)) \subseteq \Gal(K_N/k(t))$, so it suffices to show that
$\Gal(K_N\cdot k' / k'(t)) \cong [S_d]^N$ for some extension $k'$ of
$k$.  Hence, we may assume that $k$ is algebraically closed.  Since
$k$ is then infinite, and a change of variables on $\varphi$ does not
affect $\Gal(K_N/k(t))$, we may assume that  if 
$m \leq N$, then $\varphi^m(a)$ is not the point at
infinity.
Furthermore, we may assume that every prime in $k[t]$ is of the form
$(z - t)$ for some $z \in k$.

\begin{lemma}\label{disjoint ramification} For $n < N$, the prime $(\varphi(a) - \alpha_i)$ of $k[\alpha_i]$ does not ramify in $\widehat{M_i}$.
\end{lemma}

\begin{proof}
	We will show that
	$(\varphi(a)-\alpha_i)$ does not ramify in $K_n$ and
	that the primes extending $(\varphi(a)-\alpha_i)$ in $K_n$ do not ramify in
	$M_jK_n$ if $i \neq j$.
	
	We have assumed that $\varphi^{n+1}(a)-t \neq \varphi^m(b)-t$ for any
	$m \leq n$ and any critical point $b\neq a$ of $\varphi$.  Thus, we see that $(\varphi^{n+1}(a) - t)$ does not ramify in $K_n$ since the only primes of $k(t)$ that ramify in $K_n$ must divide
	
	\[\Delta(\varphi^n(x)-t)=\prod_{b\in \varphi_\mathfrak{c}}\left((\varphi(b)-t)^{d^{n-1}}(\varphi^2(b)-t)^{d^{n-2}}\dots(\varphi^n(b)-t)\right)^{e(b|\varphi(b))-1}.\]

	Since $(\varphi(a)-\alpha_i)$ extends $(\varphi^{n+1}(a) - t)$
	in $k(\alpha_i)/k(t)$, it follows that $(\varphi(a)-\alpha_i)$ does not ramify in $K_n$.
	
	We can also see that that $(\varphi(a) - \alpha_i)$ does  not ramify in $M_jK_n$ for $j \neq i$ since the primes of $K_n$ ramifying in $M_jK_n$ are those
	dividing
	\[\Delta(\varphi(x)-\alpha_j)=\prod_{b\in \varphi_\mathfrak{c}}
	(\varphi(b)-\alpha_j)^{e(b|\varphi(b))-1}.\] 
	Suppose a
	prime $\fp$ of $K_n$ extending $(\varphi(a)-\alpha_i)$ in
	$K_n/k(\alpha_i)$ ramifies in $M_jK_n$. Then $\fp$ divides
	$\Delta(\varphi(x)-\alpha_j)$, so $\fp$ divides $(\varphi(b)-\alpha_j)$ for some critical point $b$ of $\varphi$.  Hence, $\fp$ divides $(\varphi(a) -
	\alpha_i)$ and $(\varphi(b) - \alpha_j)$.  Thus, $\fp$ divides $(\varphi^{n+1}(a) - t)$ and
	$(\varphi^{n+1}(b) - t)$, so we must have
	$\varphi^{n+1}(a) = \varphi^{n+1}(b)$ (since $\fp$ can extend exactly one prime in $K_n/k(t)$).  This means that $b = a$, since we assumed $\varphi^{n+1}(a) \neq \varphi^{n+1}(b)$ if $b\neq a$.  Thus, $\fp$ divides both $(\varphi(a)- \alpha_i)$ and $(\varphi(a) - \alpha_j)$. Then
	$\fp^2$ divides $\varphi^{n+1}(a)-t=\prod_{i=1}^{d^n}(\varphi(a)-\alpha_i)$. So the prime $\fp$ of $K_n$ ramifies over $(\varphi^{n+1}(a) - t)$, which is a contradiction.
\end{proof}

\begin{proof}[Proof of Theorem \ref{sufficient}] We use induction on $n$ to prove $\Gal(\varphi^n(x)-t/k(t))\cong [S_d]^n$ for all $n\leq N$.  The result holds in the case $n=1$ by hypothesis.
	
	Let $n<N$, and suppose $\Gal(\varphi^m(x)-t/k(t)) \cong [S_d]^n$, for all $m\leq n$. Let $\alpha_1,\ldots,\alpha_{d^n}$ be the distinct roots of $\varphi^n(x)-t$ as before. Then since $\alpha_i$ is transcendental over $k$, $\Gal(M_i/k(\alpha_i))\cong \Gal(\varphi(x)-t/k(t)) \cong S_d$ where $M_i$ is the splitting field of $\varphi(x)-\alpha_i$ over $k(\alpha_i)=k(\alpha_i,t)$.
	
	Let $K_{n+1}$ be the splitting field of $\varphi^{n+1}(x)-t$ over $k(t)$, so $K_{n+1}=\prod M_j$. To complete the proof it is enough to show that $\Gal(M_i/M_i\cap\widehat{M_i}) \cong S_d$ for each $i$, as then, $\Gal(K_{n+1}/\widehat{M_i})\cong S_d$ for each $i$, and this implies that $K_{n+1}$ has degree $(d!)^{d^n}$ over $K_n$ and $[K_{n+1}:K_n][K_n:k(t)]=(d!)^{d^n}|[S_d]^n|=|[S_d]^{n+1}|$. Since $\Gal(K_{n+1}/k(t))$ must be isomorphic to a subgroup of $[S_d]^{n+1}$, we have equality.
	
	Note, the extension $M_i\cap \widehat{M_i}/k(\alpha_i)$ is Galois, so $\Gamma:=\Gal(M_i/M_i\cap \widehat{M_i})$ is a normal subgroup of $\Gal(M_i/k(\alpha_i))\cong S_d$. So either $\Gamma \cong S_d$, or $\Gamma$ is isomorphic to a subgroup of $A_d$. Let $\fp$ be the prime $(\varphi(a)-\alpha_i)$ of $k[\alpha_i]$. If $p\neq2$, then $\fp||\Delta(\varphi(x)-\alpha_i)=\prod_{b\in \varphi_\mathfrak{c}}(\varphi(b)-\alpha_i)^{e(b|\varphi(b))-1}$. Thus, if $\fq$ is any prime of $M_i$ lying over $\fp$, then by Lemma \ref{singletransposition}, $I(\fq|\fp)$ consists of a single transposition. If $p=2$, then by hypothesis, $I(\fq|\fp)$ consists of a single transposition. Now fix a prime $\fq$ of $M_i$ lying over $\fp$, and let $\fp':=\fq\cap (M_i\cap\widehat{M_i})$.  By Lemma \ref{disjoint ramification}, we see that $\fp$ does not ramify in $\widehat{M_i}$ which implies $\fp'$ is unramified over $\fp$. Hence, $e(\fq|\fp')=e(\fq|\fp)=2$, which implies $I(\fq|\fp')$ also consists of a single transposition. Thus, $\Gamma$ contains a transposition, so $\Gamma\not\subseteq A_d$ and we have $\Gamma \cong S_d$ as desired.
\end{proof}

Next, we show that as long as $(d,p)\neq(2,2)$, the conditions in Theorem \ref{sufficient} are not too restrictive and in fact ``most'' polynomials satisfy the more restrictive conditions listed below. For the rest of this section let $k$ be algebraically closed.

\begin{definition} \label{defH} Define $H(d,N,k)$ to be the set of all $f(x) \in k[x]$ such that 
	\begin{enumerate}
		\item \label{2} $f'(x)$ is separable if $p \neq 2$ and $f'(x)$ is the square of a separable polynomial if $p = 2$,
		\item \label{3} if $w_1,\ldots,w_r$ are the distinct  critical points of $f(x)$ then $f^n(w_i)\neq f^m(w_j)$ for all $1\leq i,j, \leq r$ and $m,n\leq N$, unless $m=n$ and  $i=j$.
	\end{enumerate}
	
	If $p= 2$ or $p|d$, we impose further conditions;
	\begin{itemize}
		\item[(3)] \label{4} if $p=2$, whenever $b,c \in k, (x-b)^3$ does not divide $f(x)-c$ in $k[x]$, and
		\item[(4)] \label{5} if $p|d$,  $f(x)$ is indecomposable in $k(x)$ and $\deg f'(x) = d-2$.
	\end{itemize}
\end{definition}

\begin{remark} We could impose condition (3) and the indecomposability condition in any characteristic to get an appropriate Zariski-open set. However, they are unnecessary in the cases not listed above so we choose not to do so.
\end{remark}

\begin{lemma}\label{Zariski}  $H(d,N,k)$ is a nonempty Zariski-open subset of $\cP_d(k)$.
\end{lemma}

\begin{proof}
	Let $H_N$ be the set of all $f(x) \in \cP_d(k)$  satisfying (1) and (2) above and if $p|d$ satisfying $\deg f'(x) = d-2$. Let $x, y_1, y_2, \ldots,y_r,u_0,u_1,\ldots,u_d,v$ be algebraically independent variables over $k$, where $r=d-1$ if $p\neq2$ and $p\nmid d$, $r=d-2$ if $p\neq 2$ and $p|d$, $r=\frac{d-1}{2}$ if $p=2$ and $p\nmid d$, and $r=\frac{d-2}{2}$ if $p=2$ and $p|d$. Define
	\[F(x,u_0,\ldots,u_d)=\sum_{i=0}^d u_ix^i,\]
	\[G(x,v,y_1,\ldots,y_r)=v\prod_{i=1}^r (x-y_i).\]
	If $\sigma_0,\sigma_1,\ldots,\sigma_r$ are the elementary symmetric polynomials in $y_1,\ldots,y_r$ and we set $v_i=v\sigma_i$, then $u_0,u_1,\ldots,u_d, v_0,v_1,\ldots,v_r$ are algebraically independent over $k$. 
	Let $F^m(x)=F^m(x,u_0,\ldots.,u_d)$ be the $m$-th $x$-iterate of $F(x,u_0,\ldots,u_d)$, and let
	\[D=\prod\prod_{1\leq i< j \leq r}\prod\prod_{0\leq \ell, m \leq N} [F^\ell(y_i)-F^m(y_j)].  \]
	Then $D$ is expressible as a polynomial in $u_0,\ldots,u_d,v_0,\ldots,v_r$. If $p\neq 2$, we specialize $G(x)$ to $F'(x)$, that is, specialize the $v_i$ so that $\sum_j ju_jx^{j-1}=\sum_i v_{r-i}x^i$. If $p=2$, then we let $D$ as above but specialize the $v_j$ so that $\sum_j  ju_jx^{j-1}=(\sum v_{r-i}x^i)^2=\sum v_{r-i}^2x^{2i}$. In either case, $D$ specializes to a polynomial $h(u_0,\ldots,u_d)$ in the ring $k[u_1,\ldots,u_d]$. It is clear that if $h(a_0,\ldots,a_d)\neq 0$ then $f(x)=\sum_i a_ix^i \in H_N$ so $H_N$ is a Zariski-open set in $\cP_d(k)$.

	We now show that $H_N$ is nonempty. Let $H_1$ be the set of all $f(x)\in \cP_d(k)$ satisfying 
	\begin{itemize}
		\item if $p|d$, $\deg f'(x)= d-2$,
		\item $f'(x)$ is separable if $p \neq 2$ and $f'(x)$ is the square of a separable polynomial if $p = 2$, and
		\item if $w_1,\ldots,w_r$ are the distinct critical points of $f(x)$ then $f(w_i)\neq f(w_j)$ unless $i=j$.
		
	\end{itemize}
	
	We first show that $H_1$ is nonempty, to do so we consider many different cases. 
	If $p=0$, $p \nmid d(d-1)$, or $p=2$ and $d \equiv 3 \mod 4$, then any $f(x)$ of the form $f(x)=a_dx^d+a_1x+a_0$ with $a_0a_1a_d \neq 0$, belongs to $H_1$. In the case, $p>2$ and $p|d$, we have, $d \geq p \geq 3$ and $p\nmid d-2$. If $f(x)=a_dx^d+a_{d-1}x^{d-1}+a_1x+a_0$ with $a_0a_1a_{d-1}a_d \neq 0$, then $f(x) \in H_1$.
	If $p>2$ and $p|d-1$, then $d\geq p+1\geq 4$, and  any $f(x)=a_dx^d+a_2x^2+a_0$ with $a_0a_2a_d \neq 0$, belongs to $H_1$.
	
	Finally, we consider the cases where $p=2$ and $d\not\equiv 3 \mod 4$. If $d\equiv 0\mod 4$ and $f(x) =a_dx^d+a_{d-1}x^{d-1}+a_1x+a_0$ with $a_0a_1a_{d-1}a_d \neq 0$, then $f(x)\in H_1$. If $d \equiv 1 \mod 4$ then $f(x)=a_dx^d+a_3x^3+a_0$ with $a_0a_3a_d\neq 0$ will lie in $H_1$. If $d\equiv 2 \mod 4$, $f(x)\in H_1$ if $f(x)=a_dx^d+a_{d-1}x^{d-1}+a_3x^3+a_0$ with $a_0a_3a_{d-1}a_d \neq 0$.
	
	Now we will show that $H_N$ is nonempty by showing there is some $f(x) \in H_1$ such that  $f(x)$ satisfies condition (2). Let $f(x)\in H_1$, and let $\lambda , \mu \in k$ with $\mu \neq 0$.  It is easy to see that $f^*(x)=f(\mu x+\lambda)$ also lies in $H_1$,  since if $w_i$ is a critical point of $f$ then $\frac{w_i-\lambda}{\mu}$ is a critical point of $f^*$ and $f^*\left(\frac{w_i-\lambda}{\mu}\right)=f(w_i)$.  Now note, for each $i,j,n,m$ with $i\neq j$ and $m\neq n$ the set of $(\lambda, \mu)\in \bA^2$ such that $(f^*)^n\left(\frac{w_i-\lambda}{\mu}\right)=(f^*)^m\left(\frac{w_j-\lambda}{\mu}\right)$ is a dimension one subvariety of $\bA^2$. Thus, since $k$ is infinite, we may choose $(\lambda ,\mu)\in \bA^2(k)$ so that  $(f^*)^n\left(\frac{w_i-\lambda}{\mu}\right)\neq (f^*)^m\left(\frac{w_j-\lambda}{\mu}\right)$,  for all $1\leq i,j \leq r$ and  $m,n\leq N$, unless $m=n$ and $i=j$, which implies $f^*(x) \in H_N$, and $H_N$ is nonempty.
	
	Now, if $p=2$, consider the set of $f(x) \in \cP_d(k)$ satisfying condition (3). Let $s,v,u_0,\ldots,u_{d-3},y_0,\ldots,y_d,x$ be algebraically independent variables over $k$, and define $\pi_0,\ldots,\pi_d \in k[s,v,u_0,\ldots,u_{d-3}]$ so that 
	\[\sum_{j=1}^d \pi_jx^j=s+(x-v)^3(u_0+\ldots+u_{d-3}x^{d-3}).\]
	
	Working in the ring $R:=k[s,v,u_0,\ldots u_{d-3},y_0,\ldots,y_d]$, let $\fP$ be the ideal generated by $y_0-\pi_0,\ldots,y_d-\pi_d$. Clearly, $R/\fP \cong k[s,v,u_0,\ldots,u_{d-3}]$, so $\fP$ is a prime ideal. Then $\fp=\fP\cap k[y_0,\ldots,y_d]$ is prime in $k[y_0,\ldots,y_d]$. Moreover, the transcendence degree of $k[y_0,\ldots,y_d]/\fp$ over $k$ does not exceed $d$ since $k[y_0,\ldots,y_d]/\fp \subseteq R/\fP$. Let $\cV$ be the variety in $\bA_k^{d+1}$ corresponding to $\fp$. Then $\cV$ is Zariski closed and not equal to $\bA_k^{d+1}$. It is clear that if $f(x)=a_dx^d+\ldots+a_0$ fails to satisfy the third property then $(a_0,\ldots,a_d) \in \cV$. Thus, the set $\cV^c$ is a nonempty Zariski-open set on which condition (3) holds.
	
	Finally, suppose $p|d$, it remains to show that the set of indecomposable polynomials with degree $d$ contains a nonempty Zariski-open set. It suffices to show that for each ordered pair $(e,f)\in \bN^2$ with $e,f\geq 2$ and $ef=d$, the set of polynomials in $\cP_d(k)$ that can be expressed as $g(h(x))$ in $k[x]$ with $\deg g = e$ and $\deg h= f$ is contained in a proper Zariski closed set. If $d$ is prime the result is trivial. 
	
	First, we assume $d\geq 6$ leaving the case $d=4$ until later. Note that whenever $f(x)=g(h(x))$ we can adjust $g(x)$ and $h(x)$ so that $h(x)$ is monic. Let  $x,y_0,\ldots,y_d,s_0,\ldots,s_e,t_0,\ldots,t_{f-1}$ be algebraically independent variables over $k$ and define $\pi_0,\ldots,\pi_d \in k[s_0,\ldots,s_e,t_0,\ldots,t_{f-1}]$ so that
	\[\sum_{j=0}^d \pi_j x^j=\sum_{i=0}^e s_i \left(x^f+\sum_{\ell=0}^{e-1} t_\ell x^\ell\right)^i.\]
	
	Let $\fP$ be the ideal in $R:=k[y_o,\ldots,y_d,s_0,\ldots,s_e,t_0,\ldots,t_{f-1}]$ generated by $y_0-\pi_0,\ldots,y_d-\pi_d$. Then $R/\fP\cong k[s_0,\ldots,s_e,t_0,\ldots,t_{f-1}]$, so $\fP$ is a prime ideal. Then $\fp=\fP\cap k[y_0,\ldots,y_d]$ is prime in $k[y_0,\ldots,y_d]$ and the transcendence degree of $k[y_0,\ldots,y_d]/\fp$ over $k$ is less that or equal to  $e+f+1$. Let $\cW$ be the variety in $\bA_k^{d+1}$ corresponding to $\fp$. Then $\dim \cW \leq e+f+1<ef+1=d+1$ since $d>4$. Thus, $\cW$ is a proper, Zariski closed subset of $\bA_k^{d+1}$ and clearly, if $f(x)=a_dx^d+\ldots+a_0$ is decomposable as $g(h(x))$ with $\deg g =e$ and $\deg h=f$, then $(a_0,\ldots,a_d)\in \cW$.
	
	Now consider the case $d=4$ and $e=f=2$. Note, if $\ch k=2$ and $a_4a_3\neq 0$, then it is easy to see that $f(x)=a_4x^4+\ldots+a_0$ is indecomposable. If $\ch k\neq 2$ and $f$ is decomposable, then by completing the square we can write $f(x)=g((x-c)^2)$ with $c\in k$ and $\deg g =2$. Then $f(c+x)=f(c-x)$. So if we write $f(x)=a_4x^4+a_3x^3+a_2x^2+a_1x+a_0$, expand $f(c+x)-f(c-x)=0$, and examine the coefficients we see that 
	\[4a_4c+a_3=0=4a_4c^3+3a_3c^2+2a_2c+a_1.\]
	Then since $a_4\neq 0$, we must have $c=-a_3/4a_4$, so that \[16a_1a_4^2-8a_2a_3a_4+3a_3^3a_4-a_3^3=0.\] Clearly this does not hold for all $f(x)\in\cP_4(k)$, and this completes the proof.
	
\end{proof}

\begin{theorem} \label{galois group} If $f(x)\in H(d,N,k)$, then $\Gal(f^N(x)-t/k(t)) \cong [S_d]^N$.
\end{theorem}

\begin{proof} We will show that $f(x)$ satisfies the hypotheses of Theorem \ref{sufficient} which gives the desired result. First note that for any critical point $a$ of $f(x)$, $f^n(a)\neq f^n(b)$ for all $m\leq n\leq N$, unless $m=n$ and $b=a$ by definition. Also, if $p\neq 2$ then each critical point has multiplicity one.
	
	If $p=2$, then $\Delta(f(x)-t)=\prod_{i=1}^r (f(w_i)-t)^2$. So Corollary \ref{transposition or three cycle} implies that $I(\fq|\fp)$ consists of a transposition or a three cycle for any ramified prime $\fp=(f(w_i)-t)$ and any prime $\fq$ of $K_1$ lying over $\fp$.  Recall, $K_1$ is defined to be the splitting field of $f(x)-t$ over $k(t)$. Now, condition (3) in the definition of $H(d,N,k)$ implies that the reduction of $f(x)-t \mod \fp$ is cube free. Which, by Kummer's Theorem (see \cite[Theorem 7.4]{Jan}, for example), implies that we cannot have $\fp k[\alpha]= \fP_1^3\fP_2\ldots\fP_m$, where $\alpha$ is a root of $f(x)-t$. So $I(\fq|\fp)$ cannot consist of a three cycle and we must have $I(\fq|\fp)$ consists of a single transposition.

	It remains to show $\Gal(f(x)-t/k(t))\cong S_d$. Note, for any ramified prime $\fp$ we now have $I(\fq|\fp)$ consists of a single transposition. First, we consider the case $p\nmid d$. Let $I \subseteq G$ be the subgroup generated by $\{I(\fq|\fp): \fq|\fp, \fq\in \bP_{K_1}, \text{ and } \fp\in \bP_{k(t)}\setminus\{\fp_\infty\}\}$. Then $K_1^I$ is unramified over all primes of $k[t]$, so by Lemma \ref{riemann hurwitz}, $K_1^I=k[t]$. Thus, $G=I$. So $G$ is a transitive subgroup of $S_d$ generated by transpositions and Lemma \ref{transitive subgroup} implies $G \cong S_d$.

	If $p|d$, then property (4) in the definition of $H(d,N,k)$ guarantees that $f(x)$, and hence $f(x)-t$ is indecomposable. Then since $G$ contains a transposition, Lemma  \ref{indecomposable} and Lemma \ref{primitive} imply $G\cong S_d$, as desired.
\end{proof}


\section{Generic Polynomials and Generic Rational Functions} \label{Generic}

In this section we prove Theorem \ref{generic} for $(d,p)\neq(2,2)$; we handle the case $d=p=2$ in Section \ref{2,2}.  First we give a lemma, which can be found in \cite{odoni}.

\begin{lemma}[\cite{odoni}, Lemma 2.4] \label{specialization}  Let $A$  be an integrally closed domain with field of fractions $K$, let $K'$ be any field, and let $\psi:A\lra K'$ be a ring homomorphism. Define $\widetilde{\psi} :A[x] \lra K'[x]$ by $a_dx^d+a_{d-1}x^{d-1}+\dots+a_0 \mapsto \psi(a_d)x^d+\psi(a_{d-1})x^{d-1}+\dots+\psi(a_0)$. If $f(x)=a_dx^d+\dots+a_0$ is a polynomial in $A[x]$ with $d\geq 1$, $a_d \neq 0$, and $a_d \notin \ker(\psi)$, such that $\widetilde{\psi}(f(x))$ is separable over $K'$ then $f(x)$ is separable over $K$ and $\Gal(\widetilde{\psi}(f(x))/K')$ is isomorphic to a subgroup of $\Gal(f(x)/K)$. 
\end{lemma}

\begin{proof}[Proof of Theorem \ref{generic}] We now prove Theorem \ref{generic} in the case $(d,p)\neq(2,2)$. The case $(d,p)=(2,2)$ is handled in Section \ref{2,2}, Corollary \ref{generic 2,2}. 
	
	The result follows almost immediately from Theorem \ref{galois group}.  Let $k$ be any field (not necessarily algebraically closed). Let $f(x) \in H(d,n,\overline{k})$. If $b \in \overline{k}$ then it is easy to see that $f^*(x)=b^{-1}f(bx)\in H(d,n,\overline{k})$ so without loss of generality we can assume $f(x)$ is monic.  By Lemma \ref{galois group},  $\Gal(f^n(x)-t/\overline{k}(t)) \cong[S_d]^n$. Now define $g(x):=f(x+t)-t$. Then, $g^n(x)=f^n(x+t)-t$, so $\Gal(g^n(x)/\overline{k}(t)) \cong \Gal(f^n(x)-t/\overline{k}(t)) \cong[S_d]^n$. 
	
	Now consider the maps $\psi_1:\overline{k}[\textbf{s}]\lra \overline{k}(t)$ and $\psi_2:\overline{k}[\textbf{s,u}] \lra \overline{k}(t)$, given by mapping $s_i$ to the $i$-th coefficient of $g(x)$ and mapping $u_0$ to $1$ and $u_i$ to $0$ for $i\neq 0$. We can extend $\psi_1$ and $\psi_2$ to $\widetilde{\psi_1}:\overline{k}[\textbf{s}][x]\lra \overline{k}(t)[x]$ and $\widetilde{\psi_2}:\overline{k}[\textbf{s,u}][x] \lra \overline{k}(t)[x]$ in the natural way. Let $P_n(x)$ be the numerator of $\Phi^n(x)$ then $\widetilde{\psi_1}(\fG^n(x)) =\widetilde{\psi_2}(P_n(x))= g^n(x)$, so Lemma \ref{specialization} implies $\Gal(\fG^n(x)/\overline{k}(\textbf{s})) \supseteq [S_d]^n$ and 
	$\Gal(\Phi^n(x)/\overline{k}(\textbf{s,u})) \supseteq [S_d]^n$. On the other hand, by Corollary \ref{iterated wreath}, $\Gal(\fG^n(x)/{k}(\textbf{s})) \subseteq [S_d]^n$ and 
	$\Gal(\Phi^n(x)/k(\textbf{s,u})) \subseteq [S_d]^n$. Thus, we have $[S_d]^n\subseteq \Gal(\fG^n(x)/\overline{k}(\textbf{s}))\subseteq \Gal(\fG^n(x)/{k}(\textbf{s}))\subseteq [S_d]^n$ and $[S_d]^n\subseteq \Gal(\Phi^n(x)/\overline{k}(\textbf{s,u}))\subseteq \Gal(\Phi^n(x)/{k}(\textbf{s,u}))\subseteq [S_d]^n$. Hence, it follows that $\Gal(\fG^n(x)/k(\textbf{s}))\cong [S_d]^n$ and $\Gal(\Phi^n(x)/k(\textbf{s,u}))\cong [S_d]^n$ as desired.
\end{proof}

Let $k$ be any field, let $t_1,\ldots t_r, x_1,\ldots,x_n$ be independent indeterminants over $k$, let $f_1(\textbf{x},\textbf{t}), \ldots f_m(\textbf{x},\textbf{t})$ be irreducible  polynomials in $\textbf{x}$ with coefficients in $k(\textbf{t})$, and let $g(\textbf{t})\in k[\textbf{t}]$ be a nonzero polynomial. We define the subset $\cH_k(f_1,...,f_m;g)$ of $k^r$ to be the set of all $\textbf{a}=(a_1,\ldots,a_r)\in k^r$ such that $f_i(\textbf{a},\textbf{x})$ is irreducible for all $i$ and $g(\textbf{a})\neq 0$. A \textit{Hilbert subset} of $k^r$ is any subset of this form. 

If every Hilbert subset of $k^r$ is nonempty for every integer $r\geq 1$, then we say that $k$ is a \textit{Hilbertian field}. In a Hilbertian field any Hilbert subset of $k^r$ is Zariski dense in $k^r$ (see \cite{FJ}, for example).

The following is a generalization of \cite[Lemma 6.1]{odoni}.

\begin{lemma} \label{hilbertian} Let $k$ be a Hilbertian field, and let $t_1,\ldots,t_r, x$ be independent indeterminants over $k$. Suppose that in $k[\textbf{t},x]$,  $f(\textbf{t},x)$ is $x$-monic, irreducible, and separable. Then, there is a Hilbert subset $\cH$ of $k^r$ such that for all $\textbf{t'} \in \cH$, $f(\textbf{t'},x)$ is monic, irreducible, and separable in $k[x]$ and $\Gal(f(\textbf{t'},x)/k)\cong \Gal(f(\textbf{t},x)/k(\textbf{t}))$.
\end{lemma}

\begin{corollary} Let $k$ be a Hilbertian field and $d>1$ an integer with $(d,\ch k)\neq (2,2)$. Then there are a Hilbert subsets $\cH_1$ and $\cH_2$  of $k^{d}$ and $k^{2d-1}$ respectively, such that $\Gal(f^n(x)/k)\cong [S_d]^n$ for any $f(x)=x^d+a_{d-1}x^{d-1}+\ldots+a_0$ with $(a_{d-1},\ldots, a_0)\in \cH_1$, and $\Gal(\varphi^n(x)/k)\cong [S_d]^n$ for any $\varphi(x)=\frac{x^d+a_{d-1}x^{d-1}+\ldots+a_0}{b_dx^d+b_{d-1}x^{d-1}\ldots+b_0}$ with $(a_{d-1},\ldots, a_0, b_d,\ldots, b_0)\in \cH_2$.
\end{corollary}


\section{The Case $(d,p)=(2,2)$}\label{2,2}

In the case $d=p=2$, we get different results for polynomials and rational functions so we examine these cases separately. First we look at rational functions since this case is much like the cases we have already examined.

\subsection{Rational Functions}

\begin{theorem} Let $k$ be an algebraically closed field with characteristic $2$. For any $N\in \bN$, there is a nonempty Zariski-open subset, $H$, of $\Rat_2(k)$ such that for any  $\varphi(x)\in H$, $\Gal(\varphi^N(x)-t/k(t)) \cong [S_2]^N$.
\end{theorem}

\begin{proof}
	Let $H$ be the set of degree two rational functions with coefficients in $k$ such that 
	\begin{itemize}
		\item $\varphi(x)$ has a finite critical point $w$, and
		\item $\varphi^m(w)\neq \varphi^n(w)$, for $0\leq n,m \leq N$, unless $n= m$. 
	\end{itemize}
	
	If $\varphi(x)=\frac{a_2x^2+a_1x+a_0}{b_2x^2+b_1x+b_0}$ then $\varphi'(x)=\frac{(a_1b_2-a_2b_1)x^2+(a_1b_0-a_0b_1)}{(b_2x^2+b_1x+b_0)^2}$, so $\varphi$ has a finite critical point if $a_1b_2-a_2b_1\neq 0$. Using similar arguments to those in the proof of Lemma \ref{Zariski}, we see that $H$ is Zariski-open. To see that $H$ is nonempty, note that if $\varphi(x)=\frac{x^2+a_1x+a_0}{b_1x}$ where $\left(\frac{a_1}{b_1}\right)^2 \neq a_0$, then the second property holds up to the first iterate. To see that there exists $\varphi(x)$ such that this property holds for any $N$, we again refer to the arguments from Lemma \ref{Zariski}.
	
	Let $\varphi(x)\in H$. Since $\varphi(w)-t$ ramifies in the splitting field of $\varphi(x)-t$, the inertia subgroup $I(\fq|\varphi(w)-t)$ is nontrivial, where $\fq$ is the prime extending $\varphi(w)-t$. Also, $I(\fq|\varphi(w)-t)$ is clearly contained in $S_2$. Thus, Theorem \ref{sufficient} implies $\Gal(\varphi^N(x)-t/k(t))\cong [S_2]^N$.
\end{proof}

Now, let $k$ be any field of characteristic $2$ and let $\Phi(x)$ be the generic rational function of degree 2 over $k$ as defined in Section \ref{Generic}. 

\begin{corollary} \label{generic 2,2} In the case $(d,p)=(2,2)$. $\Gal(\Phi^n(x)-t/k(t))\cong [S_2]^n$. 
\end{corollary}

\begin{proof} It suffices to show that the result holds in the case $k$ is algebraically closed, this follows from specializing the coefficients of $\Phi^n(x)$ to the coefficients of $\varphi^n(x+t)-t$ for any $\varphi \in H$, as in the proof of the other cases of Theorem \ref{generic}.
\end{proof}

\subsection{Polynomial Functions}

The above arguments for rational functions depend on the fact that we can find rational functions $\varphi$ for
which the critical point of $\varphi$ has an infinite orbit. However, for polynomials of degree 2 in characteristic 2, the situation is much different. In characteristic 2 any separable polynomial of degree 2 is ramified only at infinity which is a fixed point, thus, the polynomial is post-critically finite. Thus, we can expect the result to be much different in this case. The extensions obtained in this section, particularly those discussed in Theorem \ref{deg2poly} are Artin-Schreier type extensions, which are a characteristic $p$ analog to Kummer extensions, see \cite{GASH}\cite{Lang}.

Let $k$ be any field of characteristic 2, let $\fG(x)$ be the generic monic polynomial of degree 2 defined over $k$. Then $\fG(x)=x^2+sx+t$ for $s,t$ algebraically independent over $k$. 

\begin{theorem}\label{2generic} $\Gal(\fG^n(x)/k(s,t))\cong R_n \rtimes R_n^*$ where $R_n=\F_2[Y]/(Y^n)$ and $R_n \rtimes R_n^*$ is the group of invertible affine linear transformations of $R_n$. 
\end{theorem}

\begin{proof} Let $E$ be an algebraic closure of $k(s,t)$ and let $K\subset E$ be an algebraic closure of $k(s)$. Consider the surjective, $\F_2$-linear map $\cL: E \lra E$ by $\cL(\xi)=\xi^2+s\xi$. Note, every $\eta \in E$, has exactly two distinct preimages under $\cL$ in $E$, if $\cL(\xi)=\eta$, then the other preimage of $\eta$ is $\xi+s$.  It follows that $\dim_{\F_2}(\ker (\cL^n))=n$ for all $n\in \bN$. 
	Let $v_1=s$ and define a sequence $\{v_n\}_{n\in \bN}$ via $\cL(v_{n+1})=v_n$ for all $n \in \bN$. Then it is easy to see that $v_1,v_2,\ldots,v_n$ forms a basis for $\ker(\cL^n)\subseteq K$ over $\F_2$. Define $F_n=k(s,\ker(\cL^n))=k(s,v_n)$. If $\sigma$ is any $k(s)$-automorphism of $K$, then  $\sigma(\ker(\cL^n))=\ker(\cL^n)$ so $F_n/k(s)$ is a normal extension. Furthermore, for $n\geq 2$, we have $\cL^{n-1}(v_n)=v_1=s$, so $F_n/k(s)$ is finite Galois, for all $n$. 
	
	For $n\geq 2$, define $g_n(x)=\cL^{n-1}(x)-s$, then $g(v_n)=0$ where $g(x)\in k[s,x]$ is an $s$-Eisenstein polynomial of $x$-degree $2^{n-1}$. In particular, $g_n(x)$ is irreducible in $k(s)[x]$ of $x$-degree $2^{n-1}$, so that $\#\Gal(F_n/k(s))=[F_n:k(s)]=2^{n-1}$. 
	
	Now, suppose $\alpha$ and $\beta$ are zeros of $\fG^n(x)$ in the algebraic closure of $k(s,t)$. Then $\alpha+\beta \in \ker(\cL^n)$. Conversely, if $\lambda \in \ker(\cL^n)$ then $\fG^n(\alpha+\lambda)=0$. Hence, the set of $x$-zeros of $\fG^n(x)$ is precisely $\alpha +\ker(\cL^n)$, and $K_n=F_n(\alpha)$, for any root $\alpha$ of $\fG^n(x)$.
	
	Consider the specialization of $\fG(x)$ to  $\widetilde{\fG}(x)\in \overline{k}[t][x]$ given by $s\mapsto 0$. Then $\widetilde{\fG}^n(x)$ is $t$-Eisentein in $\overline{k}[t][x]$ , so it is irreducible. Thus, by Lemma \ref{specialization}, $\fG^n(x)$ is irreducible over $K(t)$, and  hence over $F_n(t)$. Fix some  root $\alpha$ of $\fG^n(x)$, then $K_n=F_n(\alpha)$. So we have $[K_n:k(s,t)]=[F_n(t,\alpha):F_n(t)][F_n(t):k(s,t)]=2^n2^{n-1}=\# (R_n \rtimes R_n^*)$. 
	
	For $\sigma \in \Gal(\fG^n(x)/k(s,t))$, if $\sigma(\alpha)=\alpha+v_\sigma$ for some $v_\sigma \in \ker \cL^n$, then for any $v\in \ker \cL^n$, we have
	\[\sigma(\alpha+v)=\sigma(\alpha)+\sigma(v)=\alpha+v_\sigma+\overline{\sigma}(v),\]
	where $\overline{\sigma}=\sigma|_{F_n}$ is the restriction of $\sigma$ to $F_n$. 
	
	Thus, $\Gal(\fG^n(x)/k(s,t))\cong B$, where $B$ is the group of all maps of the form $v\mapsto v'+\tau (v)$, where $\tau \in \Gal(F_n/k(s))$, and $v'$ is arbitrary in $\ker(\cL^n)$. We will show that $B\cong R_n \rtimes R_n^*$. Since $\ker(\cL^n)$ and $R_n$ are isomorphic as additive groups it suffices to show that $\Gal(F_n/k(s))\cong R_n^*$. Note that $\Gal(F_n/k(s))$ is uniquely determined by its action on $\ker (\cL^n)$, and since any $\tau\in\Gal(F_n/k(s))$ must commute with $\cL$, it is uniquely determined by its action on $v_n$. We will define a map $\psi$ from $\Gal(F_n/k(s))$ to $R_n$. If $\tau(v_n)=a_nv_n+a_{n-1}v_{n-1}+\ldots+a_1v_1=a_nv_n+a_{n-1}\cL(v_n)+\ldots+a_1\cL^{n-1}(v_n)$ then we must have $\tau(v)=a_nv+a_{n-1}\cL(v)+\ldots+a_1\cL^{n-1}(v)$ for all $v\in\ker(\cL^n)$. We define $\psi(\tau)=a_n+a_{n-1}y+\ldots+a_1y^{n-1}$. The map $\psi$ is clearly well-defined and injective, and is easily checked to be a homomorphism. Further, since any $\tau\in\Gal(F_n/k(s))$ is invertible, we must have $\psi(\tau)=a_n+a_{n-1}y+\ldots+a_1y^{n-1}$ where $a_n=1$. So, $\psi(\Gal(F_n/k(s)))\subseteq R_n^*$, and since $\#\Gal(F_n/k(s))=2^{n-1}=\#R_n^*$, $\psi$ maps $\Gal(F_n/k(s))$ onto $R_n^*$ and $\Gal(F_n/k(s))\cong R_n^*$ as desired.
	
\end{proof}

In many applications we wish to study the proportion of elements in a Galois group $G$ which fix some root of the polynomial, which we call the fixed point proportion of $G$, and denote it by $\FPP(G)$. One application which involves studying the fixed point proportions of Galois groups of generic iterates, is detailed in Section \ref{orbits}. In that section we will consider the fixed point proportion of the Galois groups of iterates of generic monic polynomials in the cases where $(d,p)\neq (2,2)$. We consider the case $(d,p)=(2,2)$ here for completeness.

\begin{theorem} We have \[\lim_{n\rightarrow \infty}\FPP(\Gal(\fG^n(x)/k(s,t)))=\frac 1 3.\]
\end{theorem}

\begin{proof} We have seen that $\Gal(\fG^n(x)/k(s,t))\cong B\cong R_n \rtimes R_n^*$. Moreover, if we identify  elements of $\Gal(\fG^n(x)/k(s,t))$ with elements of $B$ of the form $v'+\tau$ for $v'\in \ker(\cL^n)$ and $\tau \in \Gal(F_n/k(s))$ and identify the elements of $\ker(\cL^n)$ with elements of $R_n$ by $a_nv_n+a_{n-1}v_{n-1}+\ldots+a_1v_1 \leftrightarrow a_n+a_{n-1}y+\ldots+a_1y^{n-1}$, then it is easy to check that $v\mapsto v'+\tau(v)$ corresponds to $v\mapsto v'+\psi(\tau)\cdot v$ where $\psi$ is defined as in the previous proof and $\cdot$ is multiplication in $R_n$. Thus, $v'+\tau$ has a fixed point if and only if $v'+\psi(\tau)\cdot v=v$ for some $v$. That is, if and only if $v'=v(1-\psi(\tau))$ for some $v$, which holds if and only if $1-\psi(\tau)$ divides $v'$ in $R_n$. 
	
	If $\psi(\tau)-1=0$, then the only possible choice for $v'$ is $0$. If $\psi(\tau)-1 = y^i+\text{higher order terms}$, then $\psi(\tau)-1$ divides $v'$ if and only if $y^i$ divides $v'$. There are $2^{n-i-1}$ choices for $\psi(\tau)$ which have this form and $2^{n-i}$ such $v'$. Thus, the total number of elements of $\Gal(\fG^n(x)/k(s,t))$ which have fixed points is \[1+\sum_{i=1}^{n}2^{n-i}2^{n-i-1}=1+\frac{2^{2n-1}}{3}\left(1-\left(\frac{1}{4}\right)^{n}\right).\]
	So the fixed point proportion is 
	\[\frac{1}{2^{2n-1}}+\frac{1}{3}\left(1-\left(\frac{1}{4}\right)^{n}\right),\]
	which approaches $\frac{1}{3}$ as $n$ approaches $\infty$.
\end{proof}

We can see that unlike in the other cases $\Gal(\fG^n(x)/k(s,t))$ cannot be obtained as the Galois group of $f^n(x)-t$ over $\overline{k}[t]$ for a polynomial $f(x)\in\overline{k}[x]$ as in the other cases. In other words, there are no specializations of $s$ to $k$ preserving the Galois group such that the resulting extension of $k[t]$ is geometric.

\begin{theorem}\label{deg2poly} Let $k$ be an algebraically closed field with characteristic $2$ and let $f(x)=a_2x^2+a_1x+a_0 \in k[x]$, with $a_2a_1\neq0$ then $\Gal(f^n(x)-t/k(t))=(C_2)^n$ for all $n \in \bN$, where $C_2$ is the cyclic group of order $2$. 
\end{theorem}

\begin{proof} Clearly $f^n(x)-t$ is irreducible in $k(t)[x]$, also, since $a_1\neq 0$, $f^n(x)-t$ is separable. Let $E$ be an algebraically closed extension of $k(t)$. We can make a change of variables so that $f(x)$ is monic, thus we may assume it has the form $f(x)=x^2+ax+b$. Consider the $\F_2$-linear map $\cL:E\lra E$ defined by $\cL(\xi)=\xi^2+a\xi$. Using similar arguments to those at the beginning of the proof of Theorem \ref{2generic}, we see that $\dim_{\F_2}(\ker(\cL^n))=n$. Further, if $\alpha$ is any $x$-zero of $f^n(x)-t$ the set of $x$-zeros of $f^n(x)-t$ is precisely $\alpha +\ker(\cL^n)$. Since $k$ is algebraically closed, $\ker(\cL^n)\subset k$. Thus, the splitting field of $f^n(x)-t$ over $k(t)$ is $k(t,\alpha)$, and $\Gal(f^n(x)-t/k(t))$ has order $[k(t, \alpha):k(t)]=2^n$. 
	
	Since the splitting field of $f^n(x)-t$ is $k(t,\alpha)$, the group $\Gal(f^n(x)-t/k(t))$ is determined by its action on $\alpha$. Let $\sigma \in \Gal(f^n(x)-t/k(t))$. Then $\sigma(\alpha)=\alpha+v_\sigma$ for some $v_\sigma \in \ker(\cL^n)\subseteq k \subseteq k(t)$. 
	The map from $\Gal(f^n(x)-t/k(t))$ to the additive group $\ker(\cL^n)$ defined by $\sigma \mapsto v_\sigma$ is easily seen to be an isomorphism. Thus, $\Gal(f^n(x)-t/k(t))\cong\ker(\cL^n)\cong(C_2)^n$.   
\end{proof} 


\section{Applications}\label{applications}

\subsection{Primes Dividing Orbits}\label{orbits}

Define a \textit{global field} to be a number field or a finite extension of $\F_q(t)$ for some finite field $\F_q$. Let $k$ be a global field, let $f(x)\in k[x]$, and let $a_0 \in k$. We define the sequence $\{f^n(a_0)\}_{n\in\bN}$ and let $P_f(a_0)$ denote the set of primes of $k$ such that $v_\fp(f^n(a_0))\neq 0$ for some $n$. 
Following the work in \cite{odoni}, we show that for any $\epsilon >0$, ``most''  polynomials $f(x)\in k(x)$ satisfy $\delta_N(P_f(a_0))<\epsilon$, for any $a_0\in k$, where $\delta_N(P_f(a_0))$ is the natural density of $P_f(a_0)$.

\begin{definition} For a global field $K$, denote the set of prime ideals of $\fo_K$ by $P(K)$, let $A$ be a subset of $P(K)$, then the \textit{Dirichlet density}, $\delta_D(A)$ is defined by \[\delta_D(A):=\lim_{s\rightarrow 1^+}\frac{\sum_{\fp\in A} (N(\fp))^{-s}}{\sum_{\fp\in P(K)} (N(\fp))^{-s}}\]  and the \textit{natural density}, $\delta_N(A)$ is defined by 
	\[\delta_N(A): =\lim_{x\rightarrow \infty} \frac{\#\{\fp\in A|\Norm(\fp)<x\}}{\#\{\fp\in\bP^1(k)|\Norm(\fp)<x\}}.\] where $N(\fp)$ denotes the size of the residue field at $\fp$.
\end{definition}

\begin{theorem}[The Chebotarev Density Theorem, see \cite{FJ}\cite{LS}\cite{murty}]\label{CDT} Let $F$ be a global field, let $K$ be a finite Galois extension of $F$, and let $G=\Gal(K/F)$. For any conjugacy class $C$ of $G$, the Dirichlet density of the set of primes $\fp$ of $F$ for which $\Frob\left(\frac{K/F}{\fp}\right)= C$ exists and is equal to $\#C/\#G$. Furthermore, if $F$ is a number field or $F$ is a function field whose constant field is algebraically closed in $K$, then the natural density of this set also exists and is equal to $\#C/\#G$.
\end{theorem}

We will use the following easy lemma from \cite{odoni}.

\begin{lemma}[\cite{odoni}, Lemma 4.3] \label{FPP}   Let $\FPP([S_d]^n)$ denote the proportion of elements of $[S_d]^n$ with a fixed point. Then \[\lim_{n\rightarrow \infty}\FPP([S_d]^n) =0.\]
\end{lemma}

Since we assume $k$ is a global field, $k$ is Hilbertian and we can apply Lemma \ref{hilbertian}. 
We will make use of the following lemma in the proof of Theorem \ref{application1} to argue that if $k$ is a function field then for any $n$ there exists a Hilbert set $\cH$ in $k^{d}$ such that for all  $f(x)={x^d+c_{d-1}x^{d-1}+\ldots+c_0}$ with $(c_{d-1},\ldots,c_0)\in \cH$, $\Gal(f^{n}(x)/k)\cong [S_d]^{n}$ and the splitting field of $f^n(x)$ over $k$ is a geometric extension.

\begin{lemma}[\cite{FJ}, Corollary 12.2.3] \label{hilbertsubset} Let $L$ be a finite separable extension of a field $K$. Then every Hilbert subset of $L^r$ contains a Hilbert subset of $K^r$.
\end{lemma}

\begin{theorem}\label{application1} If $k$ is a global field $d\geq 2$, with $(d,\ch k)\neq(2,2)$ then for all $\epsilon>0$, there is a Hilbert subset $\cH$ of $k^{d}$, such that for all $f(x)=x^d+c_{d-1}x^{d-1}+\ldots+c_0$ with $(c_{d-1},\ldots,c_0)\in \cH$, $\delta_D(P_f(a_0))=\delta_N(P_f(a_0))<\epsilon$, for any $a_0\in k$.
\end{theorem}

The arguments provided here are essentially the same as those given in \cite{odoni} and \cite{Jones2}. 

\begin{proof}

	By Lemma \ref{FPP}, we can choose $n_0$ so that $\FPP([S_d]^{n_0})<\epsilon$. First suppose $k$ is a number field, 
	by Theorem \ref{generic}, $\Gal(\fG^{n_0}(x)/k(\textbf{s}))\cong [S_d]^{n_0}$. Then, by Lemma \ref{hilbertian}, there is a Hilbert subset of $k^{d}$ such that for all $f(x)={x^d+c_{d-1}x^{d-1}+\ldots+c_0}$ with $(c_{d-1},\ldots,c_0)$ in this set, $\Gal(f^{n_0}(x)/k)\cong [S_d]^{n_0}$. Define $\cH$ to be this subset. 
	
	On the other hand, if $k$ is a global function field with full field of constants $\F_q$, consider the constant field extension $k'=\F_{q^{d^{n_0}!}}\cdot k$ of $k$, we chose $k'$ in this way so that any extension of $\F_q$ contained in the splitting field of a specialization of $\fG^{n_0}(x)$ to $k[x]$ must be contained in $k'$. By Theorem \ref{generic}, $\Gal(\fG^{n_0}(x)/k'(\textbf{s}))\cong [S_d]^{n_0}$. So by Lemma \ref{hilbertian}, there is a Hilbert subset $\cH'$ of $(k')^{d}$ such that for all $f(x)={x^d+c_{d-1}x^{d-1}+\ldots+c_0}$ with $(c_{d-1},\ldots,c_0)$ in this set, $\Gal(f^{n_0}(x)/k')\cong [S_d]^{n_0}$. Now, by Lemma \ref{hilbertsubset} there is a Hilbert subset $\cH$ of $k^{d}$ such that $\cH\subset\cH'$. So for all $f(x)={x^d+c_{d-1}x^{d-1}+\ldots+c_0}$ with $(c_{d-1},\ldots,c_0)\in\cH\subset\cH'$ we have $\Gal(f^{n_0}(x)/k)\cong\Gal(f^{n_0}(x)/k')\cong [S_d]^{n_0}$. Also, if $L$ is the splitting field of $f^{n_0}(x)$ over $k$ then $L\cap\overline{k}=k$, as otherwise $L\cap k'\supsetneq k$ which would imply $[L:k]>[L:k']$ and $\Gal(f^{n_0}(x)/k)\supsetneq\Gal(f^{n_0}(x)/k')\cong[S_d]^n$ which is a contradiction. 
	
	With $\cH$ as above, 
	let $f(x)={x^d+c_{d-1}x^{d-1}+\ldots +c_0}$ for $(c_{d-1},\ldots,c_0)\in \cH$ and let $a_0 \in k$. Now, let $P_f(a_0)$ denote the set of primes of $k$ such that $v_\fp(f^n(a_0))\neq 0$ for some $n$. We split $P_f(a_0)$ into three sets, \[P_1=\{\fp:v_\fp(f^n(a_0))<0 \text{ for some } n\in\bN\},\] \[P_2=\{\fp: \fp|f(a_0)\ldots f^{n_0-1}(a_0)\Delta(f^{n_0}(x))\},\] \[P_3=\{\fp:\fp\nmid \Delta(f^{n_0}(x)), \fp|f^m(a_0) \text { for some } m\geq {n_0}\}.\] 
	
	The set $P_1$ consists of the primes for which $v_\fp(a_0)<0$ or $v_\fp(c_i)<0$ for some $i$, so clearly $P_1$ is a finite set. The set $P_2$ is also clearly finite. 
	
	Let $K_{n_0}$ denote the splitting field of $f^{n_0}(x)$ then, if $\fp \in P_3$, $\fp \nmid \Delta(f^{n_0}(x))$ so $\fp$ does not ramify in $K_{n_0}$ and the Frobenius conjugacy class $\Frob\left(\frac{K_{n_0}/k}{\fp}\right)$ is defined.  Note that if $\fp\in P_3$ then $\fp$ divides $f^m(a_0)=f^{n_0}(f^{m-n_0}(a_0))$ for some $m\geq {n_0}$. So, if $\fp\in P_3$ then $f^{n_0}(x)$ has a root mod $\fp$, which holds if and only if $f^{n_0}(x)$ has a linear factor mod $\fp$. This implies that $\fp$ has at least one prime ideal factor of residue degree one. Thus, $\Frob\left(\frac{K_{n_0}/k}{\fp}\right)$ fixes some root of $f^{n_0}(x)$. So we see that the union of the Frobenius conjugacy classes $\Frob\left(\frac{K_{n_0}/k}{\fp}\right)$ for $\fp \in P_3$ is contained in the set of elements of $[S_d]^{n_0}$ fixing at least one point. The proportion of such elements is $\FPP([S_d]^{n_0})$ defined above. Applying the Chebotarev Density Theorem (Lemma \ref{CDT})
	we see that \[\delta_N(P_3)\leq\FPP([S_d]^{n_0})<\epsilon.\]
	Then since $P_1$ and $P_2$ are finite, \[\delta_D(P_f(a_0))=\delta_N(P_f(a_0))=\delta_N(P_3)\leq\FPP([S_d]^{n_0})<\epsilon.\]
	
\end{proof}

\begin{remark} Note, the proof of Theorem \ref{application1} can be simplified if we only want to prove the result for the Dirichlet density, since in this case we can define the set $\cH$ for function fields in the same way we did for number fields.   \end{remark}

\subsection{Factorizations of Iterates over $\mathbb{F}_q$}\label{application}

Let $F$ be any field and let $f(x) \in F[x]$ be squarefree with degree $m\geq 1$. Recall, we say that a permutation in $S_m$ has \textit{cycle pattern} $(1)^{r_1}\ldots(m)^{r_m}$ if, when it is written as a product of disjoint cycles, it consists of $r_i$ cycles of length $i$ for each $1\leq i\leq m$. Note, that each $r_i$ is a non-negative integer and $\sum_i ir_i=m$. Two permutations in $S_m$ are conjugate if and only if they have the same cycle pattern.  We will discuss the concept of cycle patterns of squarefree polynomials due to Cohen \cite{Cohen2}.

\begin{definition} Let $\pi=(1)^{r_1}\ldots(m)^{r_m}$ be a cycle pattern in $S_m$. We say that a squarefree polynomial $f(x)$ of degree $m$ in $F[x]$ has \textit{cycle pattern} $\pi$ if $f(x)$ has exactly $r_i$ irreducible factors of degree $i$ for all $1\leq i\leq m$. 
\end{definition}

We have seen that the wreath power $[S_d]^n$ has a natural action on the $d$-ary rooted tree up to the $n$-th level. Labeling the branches at the top of the tree $1,\dots, d^n$, we can view $[S_d]^n$ as a subgroup of $S_{d^n}$. More precisely, we can define an injection $\iota:[S_d]^n \lra S_{d^n}$. The map $\iota$ depends only on the choice of the labeling, relabeling the tree will replace $\iota([S_d]^n)$ with a subgroup of $S_{d^n}$ that is $S_{d^n}$ conjugate to  $\iota([S_d]^n)$. Hence, for any conjugacy class $C$ in $S_{d^n}$, the size of $C\cap\iota([S_d]^n)$ is independent of the choice of $\iota$.  

Fix any $\iota$ as above. Let $\pi$ be a cycle pattern in $S_{d^n}$, and let $C$ be the conjugacy class of $S_{d^n}$ consisting of permutations with cycle pattern $\pi$. Define
\[\rho(\pi)=\#(C\cap\iota([S_d]^n))/\#[S_d]^n.\]
Then $\rho(\pi)$ is a nonnegative rational number and $\sum_\pi \rho(\pi)=1$.

Now let $\F_q$ be the finite field of order $q$ and characteristic $p$, let $b \in \F_q^*$ and let $\pi$ be a cycle pattern in $S_{d^n}$. Suppose $d \geq 2, n \geq 1$, and $(d,p)\neq (2,2)$.  Define
$A(q,b,d,n,\pi)$ to be the set of all $f(x)\in \F_q[x]$ such that 
\begin{itemize}
	\item $\deg f(x)=d$;
	\item $f(x)$ has leading coefficient $b\neq 0$;
	\item $f^n(x)$ is squarefree;
	\item $f^n(x)$ has cycle pattern $\pi$.
\end{itemize}

\begin{theorem}\label{application2}
	There is an $M=M(d,n)>0$ in $\bR$ and $q_0(d,n)\geq 2$  in $\bN$ such that 
	\[|\#A(q,b,d,n,\pi)-q^d\rho(\pi)|\leq Mq^{d-\frac{1}{2}}\]
	whenever $(d,\ch \F_q)\neq (2,2)$, $d\geq 2$, $n\geq 1$, $b\in\F_q^*$, $\pi$ is a cycle pattern in $S_{d^n}$ and $q\geq q_0$. 
\end{theorem}

\begin{proof} Let $\Omega(d,n,\overline{\F}_q)$ be the subset of $\cP_d(\overline{\F}_q)$ consisting of those $f(x)$ such that 
	\begin{itemize}
		\item $\deg(f(x)) =d$;
		\item $f^n(x)-t$ is $x$-separable over $\overline{\F}_q(t)$;
		\item $\Gal(f^n(x)-t/\overline{\F}_q(t)) \cong [S_d]^n$
	\end{itemize}

	Let $B(q,b,d,n)$ be the set of all $f(x) \in \F_q[x]$, with leading coefficient $b$ such that $f(x) \in \Omega(d,n,\overline{\F}_q)$. We first find an estimate for $\#B(q,b,d,n)$. By Theorem \ref{galois group}, $\Omega(d,n,\overline{\F}_q)$ contains a Zariski-open subset, namely $H(d,n,\overline{\F}_q)$. A more careful examination of the proof of Lemma \ref{Zariski} actually shows that there is a nonzero polynomial $\Theta(u_0,u_1,\ldots,u_d)$ with coefficients in $\F_p$ where $p=\ch \F_q$ and the total degree of $\Theta$ is bounded by some constant $C$ depending only on $d,n$, such that if $\Theta(a_0,a_1,\ldots,a_d)\neq 0$ then $f(x)=a_dx^d+a_{d-1}x^{d-1}+\ldots+a_0 \in H(d,n,\overline{\F}_q)$.
	
	The number of distinct $\beta \neq 0$ in $\overline{\F}_q$ such that $(u_d-\beta)$ divides $\Theta(u_0,\ldots,u_d)$ in $\overline{\F}_q$ is clearly bounded above by $C$. Thus, there is a subset $S$ of $\F_q^*$ with $\#S \geq (q-1)-C$ such that $\Theta(u_0,\ldots,u_{d-1},s)$ is not the zero polynomial whenever $s\in S$. 
	
	By a standard number theory argument, there is a constant $D$, depending only on $d,n$, such that for each $s \in S$ the number of $(a_0,\ldots,a_{d-1})$ in $(\F_q)^d$ for which $\Theta(a_0,\ldots,a_{d-1},s)=0$ is bounded above by $Dq^{d-1}$.  Thus, for each $s\in S$, \[\#B(q,s,d,n)\geq q^d-Dq^{d-1}.\] 
	
	We will now show that for sufficiently large $q$, the above estimate holds for all $b$ in $\F_q^*$. Let $s \in S$, $f(x) \in B(q,s,d,n)$, and $c\in \F_q^*$. It is clear that if $f(x) \in B(q,s,d,n)$ then  $c^{-1}f(cx) \in B(q,sc^{d-1},d,n)$, so the map $f(x) \mapsto c^{-1}f(cx)$ injects $B(q,s,d,n)$ into $B(q,sc^{d-1},d,n)$. Thus,
	\[\#B(q,sc^{d-1},d,n)\geq \#B(q,s,d,n).\]
	
	Now, let $H$  be the subgroup of all $d-1$ powers in $\F_q^*$ and let $e=\gcd(q-1,d-1)$, then $\#H =(q-1)/e$. Consider the set $SH\subseteq \F_q^*$, the union of all the cosets of $H$ containing elements of $S$. We will show that for sufficiently large $q$, $SH=\F_q^*$ and hence 
	
	\[\#B(q,b,d,n)\geq q^d-Dq^{d-1}\]
	for all $b\in \F_q^*$.
	
	Let $r$ be the number of distinct cosets in $SH$. Note $\#SH =r\frac{q-1}{e}$, so clearly, $r\leq e$. On the other hand, $r\#H=\#SH \geq \#S \geq (q-1)-C$. So $r \geq \frac{(q-1)-C}{\#H}=e-\frac{Ce}{q-1}$. Thus, there is some $q_0=q_0(d,n)$ such that $r=e$ for all $q\geq q_0$.

	Now, since clearly $\#B(q,b,d,n)\leq q^d$, we get the estimate 
	\begin{equation} \label{B}
	\#B(q,b,d,n)=q^d+O(q^{d-1}), \text{ for all } b \in \F_q^*,
	\end{equation}
	here the implied constant in the above equation depends only on $d$ and $n$. 
	
	Fix $b\in \F_q^*$ and $f(x) \in B(q,b,d,n)$. Then by assumption we have $\Gal(f^n(x)-t/\overline{\F}_q(t))\cong [S_d]^n$. It follows that 
	$\Gal(f^n(x)-t/{\F_q}(t))\cong [S_d]^n$ as well. Thus, the splitting field $L$ of $f^n(x)-t$ over $\F_q(t)$ is a geometric extension, that is, $L\cap \overline{\F}_q=\F_q$ and there is no extension of the constant field.
	
	Let $\pi$ be any cycle pattern in $[S_d]^n$ and let $C$ be the union of the corresponding conjugacy classes in $[S_d]^n$. Let $\alpha \in \F_q$, then $f^n(x)-\alpha$ has cycle pattern $\pi$ if and only if $\Frob\left(\frac{L/\F_q(t)}{t-\alpha}\right)$ has cycle pattern $\pi$. Applying an effective version of the Chebotarev Density Theorem for geometric function field extensions (see \cite[Proposition A.3]{CO} or \cite{murty}), we see that the number of $\alpha \in \F_q$ such that
	$f^n(x)-\alpha$ is squarefree with cycle pattern $\pi$ is 
	\begin{equation}\label{cheb}
	q\frac{\#C}{\#[S_d]^n}+O(q^{\frac{1}{2}})=q\rho(\pi)+O(q^{\frac{1}{2}}).
	\end{equation}
	Here the implied constant above depends only on $d,n$, and the genus of $L$. Since the degree of the different $D_{L/\F_q(t)}$ can be bounded above by a constant depending only on $d,n$, the Riemann-Hurwitz genus formula implies that the same is true for the genus of $L$. Hence, the implied constant in equation \ref{cheb} depends only on $d$ and $n$.
	
	Now treat $q,b,d,n$ as fixed and to shorten notation write $A(\pi)=A(q,b,d,n,\pi)$ and $B=B(q,b,d,n)$. Let $$A=\{f(x)\in \F_q[x]: \deg f(x)=d \text{ and } f(x) \text{ has leading coefficient } b\}.$$  For $\alpha \in \F_q$, let $D(\alpha, \pi)=\{f(x) \in A: f(x+\alpha)-\alpha \in A(\pi)\}$. Then 
	
	\begin{equation}\label{D}
	\sum_{\alpha \in \F_q} \#D(\alpha, \pi)=q\#A(\pi).
	\end{equation}

	We will find an estimate for $\#D(\alpha, \pi)$ and hence for $A(\pi)$ by examining the set $E(\alpha, \pi)=B\cap D(\alpha, \pi)$. Note, for $f(x)\in B$ and $\alpha \in \F_q$, the $n$-th iterate of $f(x+\alpha)-\alpha$ is $f^n(x+\alpha)-\alpha$. Hence, $f(x)$ is in $E(\alpha, \pi)$ if and only if $f^n(x+\alpha)-\alpha$ is squarefree with cycle pattern $\pi$. Since clearly, $f^n(x)-\alpha$ has the same cycle pattern as $f^n(x+\alpha)-\alpha$, we see that $f(x) \in E(\alpha, \pi)$ if and only if $f^n(x)-\alpha$ is squarefree with cycle pattern $\pi$. Hence, for fixed $f(x)$ equation (\ref{cheb}) implies,
	$\#\{\alpha \in \F_q:f(x)\in E(\alpha, \pi)\}=q\rho(\pi)+O(q^{\frac{1}{2}}), \text { for all } f(x) \in B$.
	
	Thus,  \[\sum_{\alpha \in \F_q} \#E(\alpha, \pi)=\sum_{f(x) \in B} \#\{\alpha \in \F_q:f(x)\in E(\alpha, \pi)\}= q^{d+1}\rho(\pi)+O(q^{d+\frac{1}{2}}).\]
	By equation (\ref{B}), we see that  $\#D(\alpha, \pi)=\#E(\alpha, \pi) + O(q^{\frac 1 2})$. Hence, 
	\[\sum_{\alpha \in \F_q} \#D(\alpha, \pi)= q^{d+1}\rho(\pi)+O(q^{d+\frac 1 2}).\]
	The desired result follows easily from equation (\ref{D}),
	\[\#A(\pi)=q^d\rho(\pi) + O(q^{d-\frac 1 2}).\]
\end{proof}

\begin{remark}  
	
	\begin{enumerate}
		\item  In the cases $\pi=(1)^{d^n}$ and $\pi=(d^n)^1$, corresponding to the cases where $f^n(x)$ splits completely into distinct monic factors and $f^n(x)$ is irreducible, $\rho(\pi)$ is easy to calculate. Clearly we have, $\rho((1)^{d^n})=(\#[S_d]^n)^{-1}$, while and inductive argument on $n$ gives $\rho((d^n)^1)=d^{-n}$. There is a formula due to Polya \cite{Polya}, which allows one to calculate $\rho(\pi)$ for every cycle pattern $\pi$ of $S_{d_n}$ (see \cite{Tomescu}). However, this formula is complicated.
		\item This result definitely does not hold for $\ch \overline{\F}_q=d=2$, since it is easily seen that $f^3(x)$ is always reducible in $\F_q[x]$ when $q$ is even and $\deg f =2$, whereas $\rho(\pi)>0$ for $\pi=(d^n)^1$ in $S_{d_n}$.
	\end{enumerate}
\end{remark}


\providecommand{\bysame}{\leavevmode\hbox to3em{\hrulefill}\thinspace}
\providecommand{\MR}{\relax\ifhmode\unskip\space\fi MR }
\providecommand{\MRhref}[2]{%
	\href{http://www.ams.org/mathscinet-getitem?mr=#1}{#2}
}
\providecommand{\href}[2]{#2}

\end{document}